\documentclass[12pt,leqno]{article}
\usepackage{amssymb,amsfonts,amsmath,amsthm,amscd,mathrsfs}
\usepackage{PSFigure,psfrag}
\def\R{{{{\rm l} \kern -.15em {\rm R}}}}
\def\N{{{{\rm l} \kern -.15em {\rm N}}}}
\def\E{{{{\rm l} \kern -.15em {\rm E}}}}
\def\P{{{{\rm l} \kern -.15em {\rm P}}}}
\def\D{{{{\rm l} \kern -.15em {\rm D}}}}
\def\L{{{{\rm l} \kern -.15em {\rm L}}}}
\def\Z{{{{\rm Z} \kern -.35em {\rm Z}}}}

\def\IE{{\mathbb E}}
\def\IP{{\mathbb P}}
\def\IR{{\mathbb R}}
\def\IN{{\mathbb N}}

\def\IZ{{\mathbb Z}}

\def\ss{\subset\subset}

\def\n{\noindent}
\def\dsl{\textstyle\sum\limits}

\def\dis{\displaystyle}
\def\o{\omega}
\def\fr{\mbox{\footnotesize $\dis\frac{1}{2}$}}

\def\ov{\overline}

\def\f{\footnotesize}
\def\r{\rightarrow}
\def\point{{\mbox{\large $.$}}}

\def\wt{\widetilde}

\def\cA{{\cal A}}
\def\cB{{\cal B}}

\def\cE{{\cal E}}
\def\cL{{\cal L}}
\def\cT{{\cal T}}
\def\cE{{\cal E}}
\def\cI{{\cal I}}

\def\cK{{\cal K}}
\def\cH{{\cal H}}
\def\cM{{\cal M}}

\def\cV{{\cal V}}
\def\cW{{\cal W}}
\def\cY{{\cal Y}}

\dimendef\dimen=0

\newtheorem{theorem}{Theorem}[section]
\newtheorem{lemma}[theorem]{Lemma}
\newtheorem{corollary}[theorem]{Corollary}
\newtheorem{proposition}[theorem]{Proposition}
\newtheorem{remark}[theorem]{Remark}

\begin{document}

\noindent

~

\bigskip
\begin{center}
{\bf VACANT SET OF RANDOM INTERLACEMENTS AND PERCOLATION}
\end{center}

\begin{center}
Alain-Sol Sznitman
\end{center}

\begin{center}

\end{center}

\bigskip
\begin{abstract}
We introduce a model of random interlacements made of a countable collection of doubly infinite paths on $\IZ^d$, $d \ge 3$. A non-negative parameter $u$ measures how many trajectories enter the picture. This model describes in the large $N$ limit the microscopic structure in the bulk, which arises when considering the disconnection time of a discrete cylinder $(\IZ/N\IZ)^{d-1} \times \IZ$ by simple random walk, or the set of points visited by simple random walk on the discrete torus $(\IZ/N \IZ)^d$ at  times of order $u N^d$. In particular we study the percolative properties of the vacant set left by the interlacement at level $u$, which is an infinite connected translation invariant random subset of $\IZ^d$. We introduce a critical value $u_*$ such that the vacant set percolates for $u < u_*$ and does not percolate for $u > u_*$. Our main results show that $u_*$ is finite when $d \ge 3$ and strictly positive when $d \ge 7$.
\end{abstract}

\vfill
\n
Departement Mathematik  \\
ETH Z\"urich\\
CH-8092 Z\"urich\\
Switzerland

\vfill
~
\newpage
\thispagestyle{empty}
~

\newpage
\setcounter{page}{1}

 \setcounter{section}{-1}
 
 \section{Introduction}
 \setcounter{equation}{0}
 
This article introduces a model of random interlacements consisting of a countable collection of doubly infinite trajectories on $\IZ^d$, $d \ge 3$. A certain non-negative parameter $u$ governs the amount of trajectories which enter the picture. The union of the supports of these trajectories defines the interlacement at level $u$. It is an infinite connected translation invariant random subset of $\IZ^d$.~Our main purpose is to study whether this random ``fabric'' is ``rainproof'' or not, i.e. whether its complement, the vacant set at level $u$, does not, or does contain an infinite connected component. This issue is related to the broad question ``how can random walk paths create interfaces in high dimension?''. The model we construct has a special interest because in a heuristic sense it offers a microscopic description of the ``texture in the bulk'' for two problems related to this broad question. One problem pertains to the percolative properties in the large $N$ limit of the vacant set left on the discrete torus $(\IZ/N \IZ)^d$, $d \ge 3$, by the trajectory of simple random walk with uniformly distributed starting point, up to times that are proportional to the number of sites of the discrete torus, cf.~\cite{BenjSzni06}. The other problem pertains to the large $N$ behavior of the disconnection time of a discrete cylinder $(\IZ/N\IZ)^{d-1} \times \IZ$, $d \ge 3$, by simple random walk, cf.~\cite{DembSzni06}, \cite{DembSzni08}, and also \cite{Szni08a}. In this work we establish a phase transition: for $u < u_*$, the vacant set at level $u$ does percolate, whereas for $u > u_*$, it does not. The critical value $u_*$ is shown to be non-degenerate (i.e. positive and finite), when $d \ge 7$, and finite for all $d \ge 3$. The results presented here have triggered some progress on the questions mentioned above, see in particular  \cite{SidoSzni08},  \cite{Szni07b}, \cite{Szni08b},\cite{Teix08}, \cite{Wind08}.
 
 \medskip
 We now describe the model. We consider the spaces $W_+$ and $W$ of infinite, respectively doubly infinite, nearest neighbor paths on $\IZ^d$, $d \ge 3$, that spend finite time in bounded subsets of $\IZ^d$. We denote with $P_x$, $x \in \IZ^d$, the law on $W_+$ of simple random walk starting at $x$. This is meaningful since the walk is transient in view of the assumption $d \ge 3$. We write $X_n, n \ge 0$, or $X_n, n\in \IZ$, for the canonical coordinates on $W_+$, or on $W$. We also consider the set of doubly-infinite trajectories modulo time-shifts
 \begin{equation}\label{0.1}
 \mbox{$W^* = W / \sim$, where $w \sim w^\prime$, if $w(\cdot) = w^\prime(\cdot + k)$, for some $k \in \IZ$}\,.
 \end{equation}
 
 \n
 We denote with $\pi^* : W \rightarrow W^*$, the canonical projection.
 
 \medskip
 The random interlacements are governed by a Poisson point process $\o = \sum_{i \ge 0} \delta_{(w^*_i,u_i)}$ on $W^* \times \IR_+$, with intensity measure $\nu(dw^*) du$, where $\nu$ is a certain $\sigma$-finite measure on $W^*$, which we now describe. For any finite subset $K$ of $\IZ^d$, we denote with $e_K$ the equilibrium measure of $K$,  see (\ref{1.6}) for the definition, with $W^0_K$ the subset of $W$ of trajectories entering $K$ at time $0$:
 \begin{equation}\label{0.2}
 \mbox{$W^0_K = \{w \in W; \,w(0) \in K$ and $w(n) \notin K$, for all $n < 0\}$}\,,
 \end{equation}
 
\medskip\n
 and with $W^*_K = \pi^*(W^0_K)$ the subset of $W^*$ of equivalence classes of trajectories entering $K$. We show in Theorem \ref{theo1.1} that there is a unique $\sigma$-finite measure $\nu$ on $W^*$ such that
 \begin{equation}\label{0.3}
 \mbox{$1_{W^*_K} \nu = \pi^*\circ Q_K$, for any finite subset $K$ of $\IZ^d$}\,,
 \end{equation}
 
 \smallskip\n
 where $Q_K$ is the finite measure supported on $W^0_K$ such that
 \begin{equation}\label{0.4}
 \begin{array}{rl}
 {\rm i)} & Q_K (X_0 \in \cdot) = e_K(\cdot)\,,
 \\[2ex]
 {\rm ii)} & \mbox{when $e_K(x) > 0$, conditionally on $X_0 = x$, $(X_n)_{n \ge 0}$, and $(X_{-n})_{n \ge 0}$ are}
 \\
 &\mbox{independent with respective distributions $P_x$ and $P_x$ conditioned}
 \\
 &\mbox{on $\{X_n \notin K$, for all $n \ge 1\}$}\,.
 \end{array}
 \end{equation}
 
 \n
 The motivation for such a requirement stems from Theorem 3.1 and (3.13) of \cite{BenjSzni06}, where the large $N$ limit of certain suitably defined excursions to a box of size $L < < N$, by simple random walk on $(\IZ/N\IZ)^d$ was investigated, and from the alternative characterization of $Q_K$ given in (\ref{1.26}), see also Remarks \ref{rem1.2} 2) and \ref{rem1.6} 3). Similar measures appear in \cite{Weil70} and \cite{Sil74}, p. 61, following an outline in \cite{Hunt60}. The construction we give here bypasses projective limit arguments: we instead glue together expressions for $\nu$ read in "local charts". 
 
 We denote with $\Omega$ the canonical space where $\o$ varies, cf.~(\ref{1.16}), and with $\IP$ the law turning $\o$ into a Poisson point process with intensity $\nu(d w^*) du$. The law $\IP$ enjoys a number of remarkable properties. It is invariant under translation of trajectories by a constant vector, and under time-reversal of trajectories, cf.~Proposition \ref{prop1.3}. Also when $K$ is a finite subset of $\IZ^d$, we introduce the random point process on $W_+ \times \IR_+$:
 \begin{equation}\label{0.5}
 \mu_K(\o) = \dsl_{i \ge 0} \delta_{(w_i,u_i)} \,1\{w_i^* \in W^*_K\}, \;\mbox{if} \; \o = \dsl_{i \ge 0} \delta_{(w^*_i,u_i)}\,,
 \end{equation}
 
 \n
where for $w_i^* \in W_K$, $w_i$ denotes the unique trajectory in $W_+$ starting at time $0$, where $w_i^*$ enters $K$, and following from then on $w_i^*$ step by step, cf.~(\ref{1.18}) for the precise definition. We show in Proposition \ref{prop1.3} that 
 \begin{equation}\label{0.6}
 \mbox{$\mu_K$ is a Poisson point process with intensity $P_{e_K}(dw) du$}, 
 \end{equation}
 
 \n
where $P_{e_K} = \sum_x e_K(x) P_x$. Further the point processes $\mu_K$, as $K$ varies, satisfy a compatibility condition, cf.~(\ref{1.21}), (\ref{1.46}).
 
\medskip
It may be worth pointing out that much of the above constructs, except for the aspects related to translation invariance, can be performed in the more general set-up of a transient random walk attached to an infinite locally finite connected graph with positive weights along its edges, in place of simple random walk on $\IZ^d$, $d \ge 3$, cf.~Remark \ref{rem1.4}.

\medskip
The {\it interlacement at level $u$} is defined as
\begin{equation}\label{0.7}
\cI^u(\o) = \bigcup\limits_{u_i \le u} w_i^*(\IZ), \;\mbox{if}\; \o = \dsl_{i \ge 0} \delta_{(w^*_i,u_i)} \in \Omega\,,
\end{equation}

\n
where $w_i^*(\IZ)$ denotes the range of any $w$ with $\pi^*(w) = w^*_i$. The {\it vacant set at level $u$} is then
\begin{equation}\label{0.8}
\cV^u(\o) = \IZ^d \backslash \cI^u(\o)\,.
\end{equation}

\n
Clearly $\cI^u$ increases with $u$, whereas $\cV^u$ decreases with $u$. Also one can see that the restriction of $\cI^u$ to $K$ is determined by $\mu_{K^\prime}(dw \times [0,u])$, when $K \subset K^\prime$ are finite subsets of $\IZ^d$, cf.~(\ref{1.54}). Together with (\ref{0.6}) one finds that the restriction of $\cI^u$ to $K$ can be visualized as the trace on $K$ of a Poisson cloud of finite trajectories.  Its intensity measure is proportional to the law of simple random walk run up to the last visit to $K$, with initial distribution the harmonic measure of $K$ viewed from infinity, i.e. $e_K$ normalized by its total mass cap$(K)$, the capacity of $K$, and the  proportionality factor equals $u \,{\rm cap}(K)$, cf.~Remark \ref{rem1.6} 3). We also show in Corollary \ref{cor2.3} and Proposition \ref{prop1.5} that:
\begin{equation}\label{0.9}
\mbox{$\IP$-a.s., $\cI^u$ is an infinite connected subset, and}
\end{equation}
\begin{equation}\label{0.9a}
\IP[\cV^u \supseteq K] = \exp\{- u\, {\rm cap} (K)\}\,,
\end{equation}

\n
for $u > 0$, and $K \subset \IZ^d$, finite. Formula (\ref{0.9a}) characterizes the law of $\cV^\nu$, see Remark \ref{rem2.2} 2). As a special case, cf.~(\ref{1.58}), (\ref{1.59}), one finds that for $x,y \in \IZ^d$,
\begin{equation}\label{0.10}
\IP[x \in \cV^u] = \exp \Big\{- \dis\frac{u}{g(0)}\Big\}, \;\IP[\{x,y\} \subseteq \cV^u] = \exp\Big\{-\dis\frac{2u}{g(0) + g(y - x)}\Big\}\,,
\end{equation}

\medskip\n
where $g(y-x)$ denotes the Green function, cf.~(\ref{1.5}). The identities in (\ref{0.10}) are in essence formulas (2.26) and (3.6) of Brummelhuis-Hilhorst \cite{BrumHilh91} in their theoretical physics article on the covering of a periodic lattice by a random walk, see also Remark \ref{rem1.6} 5). They display the presence of long range dependence in the random set $\cV^u$, with a correlation of the events $\{x \in \cV^u\}$ and $\{y \in \cV^u\}$ decaying as $c(u) |x-y|^{-(d-2)}$, when $|x-y|$ tends to infinity.

\medskip
As mentioned above, the main object of this work is to investigate the presence or absence of an infinite connected component in $\cV^u$. We establish in Theorem \ref{theo2.1} the ergodicity of the (properly defined) distribution of the random set $\cV^u$, from which easily follows a zero-one law for the probability of occurrence of an infinite connected component in $\cV^u$. It is then straightforward to see that this probability equals one precisely when
\begin{equation}\label{0.11}
\eta(u) \stackrel{\rm def}{=} \IP[ \mbox{$0$ belongs to an infinite connected component of $\cV^u$}] > 0\,.
\end{equation}

\n
The function $\eta(\cdot)$ is non-increasing and just as in the case of Bernoulli percolation, cf.~\cite{Grim99}, we can introduce the critical value:
\begin{equation}\label{0.12}
u_* = \inf\{u \ge 0, \eta(u) = 0\} \in [0,\infty]\,.
\end{equation}

\n
The main results of this article concern the non-degeneracy of $u_*$. We show in Theorem \ref{theo3.5} that $\cV^u$ does not percolate for large $u$, i.e.
\begin{equation}\label{0.13}
u_* < \infty, \;\mbox{for $d \ge 3$}\,,
\end{equation}

\n
and in Theorem \ref{theo4.3} that when $d \ge 7$, $\cV^u$ percolates for small $u > 0$, i.e.
\begin{equation}\label{0.14}
u_* > 0, \;\mbox{when $d \ge 7$}\,.
\end{equation}

\n
Subsequent developments initiated by the present article respectively relating random interlacements on $\IZ^{d+1}$ and on $\IZ^d$ to the local picture left by simple random walk on $(\IZ/N \IZ)^d \times \IZ$  run up to times of order $N^{2d}$, and random walk on $(\IZ/N \IZ)^d$ run up to times of order $N^d$ can be found in \cite{Szni07b}, \cite{Wind08}. In this light the results presented here with their proofs also have a bearing on the problems investigated in \cite{DembSzni06}, \cite{DembSzni08}, \cite{BenjSzni06}. In particular (\ref{0.13}) offers evidence that when $d\ge 2$ the laws of $T_N / N^{2d}$ are tight, if $T_N$ denotes the disconnection time of $(\IZ/N \IZ)^d \times \IZ$ studied in \cite{DembSzni06}, \cite{DembSzni08}. It signals that one should be able to remove the logarithmic terms present in the (very general) upper bound of \cite{Szni08a} and bypass the strategy based on the domination of $T_N$ by the cover time of $(\IZ/N \IZ)^d \times \{0\}$, by relying instead on the emergence of a non-percolative local picture of the vacant set left by random walk. These heuristic considerations can be made precise and lead to the above claimed tightness, see \cite{Szni08b}. Similarly (\ref{0.14}) offers evidence that the lower bound on $T_N$ in \cite{DembSzni08}, which shows the tightness of $N^{2d} / T_N$ when $d \ge 17$, should hold as soon as $d \ge 6$, (and in fact for all $d \ge 2$, in view of the recent work \cite{SidoSzni08}). Analogously in the context of \cite{BenjSzni06},  (\ref{0.13}), (\ref{0.14}), and \cite{SidoSzni08} give support for the typical absence for large $N$ of a giant component in the vacant set left by simple random walk on $(\IZ/N \IZ)^d$, $d \ge 3$, run up to time $uN^d$, if $u$ is large, and for its typical presence when $u$ is chosen small. 

\medskip
There are many natural questions left untouched by the present article. Is there a unique infinite component when $\cV^u$ percolates, (see Remark \ref{rem2.2} 3))? The answer is affirmative, as shown in \cite{Teix08}. Is $u_* > 0$, when $3 \le d \le 6$, as suggested by simulations? This is indeed the case, see \cite{SidoSzni08}. However it is presently unknown whether the vacant set percolates at criticality, i.e. when $u = u_*$, or what the large $d$ behavior of $u_*$ is. We refer to Remark \ref{rem4.4} 3) for further open problems.

\medskip
We will now comment on the proofs of (\ref{0.13}) and (\ref{0.14}). Most of the work goes into the proof of (\ref{0.13}). This is due to the long range dependence in the model and the fact highlighted by (\ref{0.9a}) that $\IP[\cV^u \supseteq K]$ does not decay exponentially with $|K|$. This feature creates a very serious obstruction to the Peierls-type argument commonly met in Bernoulli percolation, see \cite{Grim99}, p.~16, when one attempts to show that $\cV^u$ does not percolate for large $u$. We instead use a renormalization technique to prove (\ref{0.13}) and consider a sequence of functions on $\IR_+$:
\begin{equation}\label{0.15}
\begin{split}
p_n(u) \mbox{``$=$''} & \mbox{$\IP$-probability that $\cV^u$ contains a path from a given block}
\\
&\mbox{of side-length $L_n$ to the complement of its $L_n$-neighborhood},
\end{split}
\end{equation}

\n
(cf.~(\ref{3.8}) for the precise definition), where $L_n$ is a rapidly growing sequence of length scales, cf.~(\ref{3.1}), (\ref{3.2}),
\begin{equation}\label{0.16}
L_n \approx L_0^{(1+a)^n}, \, n \ge 0,\; \mbox{with $a = \dis\frac{1}{100d}$}\,.
\end{equation}

\n
The key control appears in Proposition \ref{prop3.1}, where we prove that for $L_0 \ge c$, $u_0 > c(L_0)$, and an increasing but bounded sequence $u_n$ depending on $L_0,u_0$, cf.~(\ref{3.9}),
\begin{equation}\label{0.17}
p_n(u_n) \underset{n \r \infty}{\longrightarrow} 0\,.
\end{equation}

\n
This immediately implies that $\eta(u) = 0$, for $u \ge u_\infty = \sup u_n(< \infty)$, and proves (\ref{0.13}). The principal difficulty in proving (\ref{0.17}) resides in the derivation of a suitable recurrence relation between $p_n(\cdot)$ and $p_{n+1}(\cdot)$, cf.~(\ref{3.52}), due to the long range dependence in the model. In a suitable sense we use a ``sprinkling technique'', where more independent paths are thrown in, so as to dominate long range dependence. This is reflected in the fact that we evaluate $p_n(\cdot)$ at an increasing but convergent sequence $u_n$ in the key control (\ref{3.10}), (a more quantitative version of (\ref{0.17})). Incidentally the sequence of length scales appearing in  (\ref{0.16}) corresponds to the choice of a small $a$,  so as to control the combinatorial complexity involved in selecting boxes of scale $L_n$ within a given box of scale $L_{n+1}$, see (\ref{3.13}), but also to the choice of a fast enough growth, so as to discard paths making more than a certain finite number of excursions at distance of order $L_{n+1}$, see below (\ref{3.25}), and (\ref{3.55}), (\ref{3.65}).

\medskip
The proof of (\ref{0.14}) in Theorem \ref{theo4.3} employs a similar albeit simpler renormalization strategy. One can instead employ a Peierls-type argument to show that $\cV^u$ percolates for small $u > 0$, when $d$ is sufficiently large, very much in the spirit of Section 2 of \cite{BenjSzni06}, or Section 1 of \cite{DembSzni08}. It is based on an exponential bound on $\IP[\cI^u \supseteq A]$, for $A$ finite subset of $\IZ^2$, (where $\IZ^2$ is viewed as a subset of $\IZ^d$), cf.~(\ref{2.23}) in Theorem \ref{theo2.4}. This estimate mirrors the exponential controls derived in Theorem 2.1 of \cite{BenjSzni06} and Theorem 1.2 of \cite{BenjSzni06}. This strategy leads to a proof of (\ref{0.14}) when:
\begin{equation}\label{0.18}
7 \Big(\dis\frac{2}{d} + \Big(1 - \dis\frac{2}{d}\Big) \;q(d-2)\Big) < 1\,,
\end{equation}

\n
with $q(\nu)$ the return probability to the origin of simple random walk on $\IZ^\nu$. In practice this means $d \ge 18$, cf.~Remark \ref{rem2.5} 3). The technique we use works instead as soon as $d \ge 7$. It also shows, just as the Peierls-type argument does when (\ref{0.18}) holds, the existence of an infinite connected cluster in $\cV^u \cap \IZ^2$, for small $u > 0$.

\medskip
Let us now describe the organization of the article.

\medskip
In Section 1 we construct the model of random interlacements. The main task lies in the construction of the $\sigma$-finite measure $\nu$ entering the intensity of the Poisson point process we are after. This is done in Theorem \ref{theo1.1}. Basic properties of the model appear in Proposition \ref{prop1.3}, whereas Proposition \ref{prop1.5} shows (\ref{0.9a}), (\ref{0.10}).

\medskip
Section 2 shows the ergodicity of the law of $\cV^u$, and the zero-one law for the probability that $\cV^u$ percolates in Theorem \ref{theo2.1}. We also prove (\ref{0.9}) in Corollary \ref{cor2.3}. In Theorem \ref{theo2.4} we derive exponential bounds that provide further link of the present model to \cite{BenjSzni06}, \cite{DembSzni08}.

\medskip
Section 3 is devoted to the proof of (\ref{0.13}) in Theorem \ref{theo3.5}. The main renormalization step is contained in Proposition \ref{prop3.1}. 

\medskip
Section 4 shows (\ref{0.14}) in Theorem \ref{theo4.3}. The principal step appears in Proposition \ref{prop4.1}.

\medskip
Finally let us state our convention regarding constants. Throughout the text $c$ or $c^\prime$ denote positive constants which solely depend on $d$, with values changing from place to place. The numbered constants $c_0,c_1,\dots$ are fixed and refer to the value at their first appearance in the text. Dependence of constants on additional parameters appears in the notation. For instance $c(L_0,u_0)$ denotes a positive constant depending on $d, L_0, u_0$.

\bigskip\bigskip\n
{\bf Acknowledgements:} We wish to thank Yves Le Jan for useful references.

 \section{Basic model and some first properties}
 \setcounter{equation}{0}

The main object of this section is to introduce the basic model and present some of its properties. As explained in the Introduction the basic model comes as a Poisson point process on a suitable state space. The main task resides in the construction of the intensity measure of this point process. This is done in Theorem \ref{theo1.1}. We then derive some of its properties in Proposition \ref{prop1.3} as well as some of the properties of the vacant set left by the interlacement at level $u$, cf.~(0.9), in Proposition \ref{prop1.5}. We first begin with some notation.

\medskip
We write $\IN = \{0,1,2,\dots\}$ for the set of natural numbers. Given a non-negative real number $a$, we write $[a]$ for the integer part of $a$, and for real numbers $b,c$, we write $b \wedge c$ and $b \vee c$ for the respective minimum and maximum of $b$ and $c$. We denote with $|\cdot |$ and $|\cdot |_\infty$ the Euclidean and $\ell^\infty$-distances on $\IZ^d$. We write $B(x,r)$ for the closed $|\cdot |_\infty$-ball with center $x$ in $\IZ^d$ and radius $r \ge 0$, and $S(x,r)$ for the corresponding $|\cdot|_\infty$-sphere with center $x$ and radius $r$, (it is empty when $r$ is not an integer). We say that $x,y$ in $\IZ^d$ are neighbors, respectively $*$-neighbors, when $|x-y| = 1$, respectively $|x - y|_\infty = 1$. The notions of connected and $*$-connected subsets are defined accordingly, and so are the notions of nearest neighbor or $*$-nearest neighbor paths in $\IZ^d$. For $A,B$ subsets in $\IZ^d$, we denote with $A+B$ the subset of elements of the form $x + y$, with $x \in A, y \in B$ and with $d(A,B) = \inf\{|x - y|_\infty$; $x \in A, y \in B\}$, the $|\cdot |_\infty$-distance from $A$ to $B$. When $U$ is a subset of $\IZ^d$, we let $|U|$ stand for the cardinality of $U$, $\partial U$ for the exterior boundary of  $U$ and $\partial_{\rm int} U$ for the interior boundary of $U$:
\begin{equation}\label{1.1}
\partial U = \{x \in U^c; \exists y \in U, |x-y| = 1\}, \;\partial_{\rm int} U = \{x \in U; \exists y \in U^c, \,|x - y| = 1\}\,.
\end{equation}

\n
We write $U \subset \subset \IZ^d$ to express that $U$ is a finite subset of $\IZ^d$. In what follows, unless otherwise explicitly mentioned, we tacitly assume that $d \ge 3$.

\medskip
We consider $W_+$ and $W$ the spaces of trajectories:
\begin{equation}\label{1.2}
\begin{split}
W_+  = \big\{ & w \in (\IZ^d)^\IN; \;|w(n+1) - w(n)| = 1, \;\mbox{for all $n \ge 0$, and}
\\
& \lim\limits_n |w(n)| = \infty\big\}\,,
\\[2ex]
W   = \big\{ & w \in (\IZ^d)^\IZ; \;|w(n+1) - w(n)| = 1, \;\mbox{for all $n  \in  \IZ$, and}
\\
& \lim\limits_{|n| \rightarrow \infty} | w(n)| = \infty\big\}\,.
\end{split}
\end{equation}

\n
We denote with $X_n, n \ge 0$, and $X_n$, $n \in \IZ$, the respective canonical coordinates on $W_+$ and $W$, and write $\theta_n$, $n \ge 0$, and $\theta_n$, $n \in \IZ$, for the respective canonical shifts. We let $\cW_+$ and $\cW$ stand for the $\sigma$-fields on $W_+$ and $W$ generated by the canonical coordinates.

\medskip
Given $U \subseteq \IZ^d$, $w \in W_+$, we denote with $H_U(w)$, $T_U(w)$, $\wt{H}_U(w)$, the entrance time in $U$, the exit time from $U$, and the hitting time of $U$ for the trajectory $w$:
\begin{equation}\label{1.3}
\begin{split}
H_U(w) & = \inf\{n \ge 0; X_n (w) \in U\}, \;T_U(w) = \inf\{n \ge 0; X_n(w) \notin U\}\,,
\\
\wt{H}_U (w)& = \inf\{n \ge 1; X_n(w) \in U\}\,.
\end{split}
\end{equation}

\medskip\n
We often drop ``$w$'' from the notation and write $H_x,T_x,\wt{H}_x$, when $U = \{x\}$. Also when $w \in W$, we define $H_U(w)$ and $T_U(w)$ in a similar fashion replacing ``$n \ge 0$'' with ``$n \in \IZ$'' in  (\ref{1.3}), and $\wt{H}_U (w)$ just as in (\ref{1.3}). For $K \subseteq U$ in $\IZ^d$, $w \in W_+$, we consider $R_k,D_k,k \ge 1$, the successive returns to $K$ and departures from $U$ of the trajectory $w$:
\begin{equation}\label{1.4}
\begin{split}
R_1 &= H_K, \; D_1 = T_U \circ \theta_{H_K} + H_K, \;\mbox{and for $k \ge 1$}\,,
\\
R_{k+1}& = R_1 \circ \theta_{D_k} + D_k, \;D_{k+1}= D_1 \circ \theta_{D_k} + D_k\,,
\end{split}
\end{equation}
so that $0 \le R_1 \le D_1 \le \dots \le R_k \le D_k \le \dots \le \infty$.

\medskip
When $X$ is an integrable random variable and $A$ an event, we routinely write $E[X,A]$ in place of $E[X\,1_A]$ in what follows, with $E$ referring here to the relevant expectation. Given $x \in \IZ^d$, we write $P_x$ for the restriction to $(W_+, \cW_+)$ of the canonical law of simple random walk on $\IZ^d$ starting at $x$. Recall that $d \ge 3$, and $W_+$ has full measure under the canonical law. When $\rho$ is a positive measure on $\IZ^d$, we write $P_\rho$ for the measure $\sum_{x \in \IZ^d} \rho(x) P_x$. We denote with $g(\cdot,\cdot)$ the Green function of the walk:
\begin{equation}\label{1.5}
g(x,y) = \dsl_{n \ge 0} P_x [X_n = y], \;x,y \in \IZ^d\,,
\end{equation}

\n
and $g(y) = g(0,y)$ so that $g(x,y) = g(y-x)$, thanks to translation invariance. Given $K \subset \subset \IZ^d$, we write $e_K$ for the equilibrium measure of $K$, cap$(K)$ for the capacity of $K$, so that, cf.~Chapter 2 \S 2 of \cite{Lawl91}:
\begin{align}
e_K(x) = & \,P_x[\wt{H}_K = \infty], \;x \in K\,, \label{1.6}
\\
& \, 0, \;\mbox{if $x \notin K$}\,, \nonumber
\intertext{(note that $e_K$ is supported on $\partial_{\rm int} K),$}
\mbox{cap}(K) = &\,e_K (\IZ^d) \,\Big(= \dsl_{x \in \IZ^d} \, e_K(x)\Big), \;\mbox{and} \label{1.7}
\\
P_x[H_K < \infty]  =& \dis\int_K g(x,y) \,e_K(dy) \Big( = \dsl_{y \in K} g(x,y) \,e_K(y)\Big), \;\mbox{for $x \in \IZ^d$}\,.\label{1.8}
\end{align}

\n
The following bounds on $P_x [H_K < \infty]$, $x \in \IZ^d$, will be useful:
\begin{equation}\label{1.9}
\dsl_{y \in K} g(x,y) / \sup\limits_{z \in K} \Big(\dsl_{y \in K} g(z,y)\Big) \le P_x [H_K < \infty] \le \dsl_{y \in K} g(x,y) / \inf\limits_{z \in K} \Big(\dsl_{y\in K} g(z,y)\Big) \,.
\end{equation}

\medskip\n
They classically follow from the $L^1(P_x)$-convergence of the bounded martingale $M_n = \sum_{y \in K} g(X_{n \wedge H_K},y)$, $n \ge 0$, towards $1\{H_K < \infty\} \sum_{y \in K} g(X_{H_K},y)$.

\bigskip
The state space of the Poisson point process we wish to define involves the quotient space $W^*$ of equivalence classes of trajectories in $W$ modulo time shift, cf.~(\ref{0.2}). We recall that $\pi^*$ stands for the canonical projection on $W^*$. We endow $W^*$ with the canonical $\sigma$-field
\begin{equation}\label{1.10}
\cW^* = \{A \subseteq W^*; \,(\pi^*)^{-1} (A) \in \cW\}\,,
\end{equation}

\n
which is the largest $\sigma$-algebra such that $(W,\cW) \overset{\pi^*}{\longrightarrow} (W^*, \cW^*)$ is measurable. When $K \subset \subset \IZ^d$, we consider
\begin{equation}\label{1.11}
W_K = \{w \in W; \;X_n(w) \in K, \;\mbox{for some $n \in \IZ$}\}\,,
\end{equation}

\n
the subset of $W$ of trajectories entering $K$. We can write $W_K \in \cW$ as a countable partition into measurable sets (see below (\ref{1.3}) for the notation):
\begin{equation}\label{1.12}
W_K  = \bigcup\limits_{n \in \IZ} \;W^n_K, \;\mbox{where} \;\;W^n_K  = \{w \in W; \,H_K(w) = n\}\,.
\end{equation}

\n
We then introduce
\begin{equation}\label{1.13}
W^*_K = \pi^* (W_K) \;\big( = \pi^* (W^0_K)\big)\,,
\end{equation}
as well as the map
\begin{equation}\label{1.14}
s_K: W^*_K \rightarrow W, \;\mbox{with $s_K(w^*) = w^0$ the unique element of $W^0_K$ with $\pi^*(w^0) = w^*$}\,.
\end{equation}

\n
Note that $s_K(W^*_K) = W^0_K$ and $s_K$ is a section of $\pi^*$ over $W^*_K$, i.e. $\pi^* \circ s_K$ is the identity map on  $W^*_K$. It is then straightforward to check that for any $K \subset \subset \IZ^d$,
\begin{equation}\label{1.15}
\mbox{$W^*_K \in \cW^*$ and the trace of $\cW^*$ on $W^*_K$ coincides with $s^{-1}_K(\cW)$}\,.
\end{equation}

\n 
There is no natural way to globally identify  $W^*\times \IZ$ with $W$, but the maps $s_K$ enable us to identify $W^*_K$ with $W^0_K$ and $W^*_K\times \IZ$ with $W_K$. In a slightly pedantic way $\pi^*: (W, \cW) \rightarrow (W^*, \cW^*)$ with the transformations $\theta_n$, $n \in \IZ$, on the fiber of $\pi^*$ could be viewed as a ``principal fiber-bundle with group $\IZ$'', cf.~\cite{Spiv79}, p.~346. The construction of the key $\sigma$-finite measure $\nu$ in Theorem \ref{theo1.1} will involve checking compatibility and patching up expressions for $\nu$ "read in the local chart $s_K$", as $K$ varies over finite subsets of $\IZ^d$.

\medskip
We further need to introduce several spaces of point measures that we will routinely use in what follows. In particular we consider $\Omega$ and $M$ the spaces of point measures on $W^* \times \IR_+$ and $W_+ \times \IR_+$:
\begin{align}
\Omega = \Big\{ & \mbox{$\o = \dsl_{i \ge 0} \delta_{(w_i^*,u_i)}$, with $(w^*_i,u_i) \in W^* \times \IR_+, \, i \ge 0$, and} \label{1.16}
\\[-0.5ex]
&\mbox{$\o(W^*_K \times [0,u]) < \infty$, for any $K \subset \subset \IZ^d, u \ge 0\Big\}$}\,,\nonumber
\\[2ex]
M = \Big\{ & \mbox{$\mu = \dsl_{i\in I} \delta_{(w_i,u_i)}$, with $I$ a variable finite or infinite subset of $\IN$,} \label{1.17}
\\[-0.5ex]
&\mbox{$ \,(w_i,u_i) \in W_+ \times \IR_+,$ for  $ i \in I$, and $\mu (W_+ \times [0,u]) < \infty$, for $u \ge 0\Big\}$}\,,\nonumber
\end{align}

\n
We endow $\Omega$ with the $\sigma$-algebra $\cA$ generated by the evaluation maps $\o \rightarrow \o(D)$, where $D$ runs over $\cW^* \otimes \cB(\IR_+)$, cf.~(\ref{1.10}). Likewise we endow $M$ with the $\sigma$-algebra $\cM$ generated by the evaluation maps $\mu \rightarrow \mu(D)$, where $D$ runs over $\cW_+ \otimes \cB(\IR_+)$, cf.~below (\ref{1.2}). Given $K \subset \subset \IZ^d$, we then define the measurable maps $\mu_K: \Omega \rightarrow M$ and $\Theta_K: M \rightarrow M$ via:
\begin{equation}\label{1.18}
\begin{split}
\mu_K(\o)(f) = &\dis\int_{W^*_K \times \IR_+} f(s_K(w^*)_+, u) \,\o(dw^*,du), \;\mbox{for $\o \in \Omega$}\,,
\\
& \mbox{and $f$ non-negative measurable on $W_+ \times \IR_+$},
\end{split}
\end{equation}

\n
where for $w \in W$, $w_+ \in W_+$ denotes the restriction of $w$ to $\IN$, so that $s_K(w^*)_+$ starts at time $0$ where $w^* \in W^*_K$ enters $K$, and follows from then on $w^*$ step by step, as well as
\begin{equation}\label{1.19}
\begin{split}
\Theta_K (\mu) (f) = & \dis\int_{\{H_K < \infty\} \times \IR_+} f(\theta_{H_K}(w), u) \,\mu(dw,du), \;\mbox{for $\mu \in M$}\,,
\\[1ex]
&\mbox{and $f$ as in (\ref{1.18})}\,,
\end{split}
\end{equation}

\n
in other words, $\Theta_K(\mu) = \sum_{i \in I} \delta_{(\theta_{H_K}(w_i),u_i)} 1\{H_{K} (w_i)  < \infty\}$, when $\mu = \sum_{i \in I} \delta_{(w_i,u_i)} \in M$. Given  $K \subset \subset \IZ^d$, $u \ge 0$, we will also consider the measurable function on $\Omega$ with values in the set of finite point measures on $(W_+, \cW_+)$:
\begin{equation}\label{1.20}
\mu_{K,u}(\o)(dw) =  \mu_K(\o) (dw \times [0,u]), \;\mbox{for $\o \in \Omega$}\,.
\end{equation}

\n
We record for later use the straightforward identities valid for $K \subset K^\prime \subset\subset \IZ^d$:
\begin{equation}\label{1.21}
\begin{array}{ll}
{\rm i)} & \Theta_K \circ \mu_{K^\prime} \,= \mu_K\,,
\\[1ex]
{\rm ii)} & \Theta_K \circ \Theta_{K^\prime} = \Theta_K \,.
\end{array}
\end{equation}

\n
We are now going to construct the $\sigma$-finite measure $\nu$ on $(W^*, \cW^*)$ which enters the intensity of the Poisson point process we wish to define. For  $K \subset \subset \IZ^d$, we write $\cT_K$ for the countable set of finite nearest-neighbor trajectories starting and ending in the support of $e_K$:
\begin{equation}\label{1.22}
\begin{array}{ll}
\cT_K = \{ \tau = \big(\tau(n)\big)_{0 \le n \le N_\tau}; &N_\tau \ge 0, \,|\tau(n+1) - \tau(n)| = 1, \; \mbox{for}\; 0 \le n < N_\tau, 
\\[1ex]
&\mbox{and $\tau(0), \tau(N_\tau) \in {\rm Supp}\,e_K\}$}\,,
\end{array}
\end{equation}

\medskip\n
if $x \in {\rm Supp} \,e_K$, we also denote with $P^K_x$ the probability on $W_+$ governing the walk conditioned not to hit  $K$:
\begin{equation}\label{1.23}
P^K_x[\cdot] = P_x[\,\cdot \,|\wt{H}_K = \infty]\,.
\end{equation}
We are now ready to state

\begin{theorem}\label{theo1.1}
For $K \subset \subset \IZ^d$, denote with $Q_K$ the finite measure on $W$, supported on $W^0_K$, such that for any $A,B \in \cW_+$, $x \in \IZ^d$:
\begin{equation}\label{1.24}
Q_K\big[(X_{-n})_{n \ge 0} \in A, \,X_0 = x, \,(X_n)_{n \ge 0} \in B] = P_x^K[A] \,e_K(x) \,P_x[B]\,.
\end{equation}

\n
There is a unique $\sigma$-finite measure $\nu$ on $(W^*, \cW^*)$ such that:
\begin{equation}\label{1.25}
1_{W^*_K} \,\nu = \pi^* \circ Q_K, \;\mbox{for any} \;K \subset \subset \IZ^d\,.
\end{equation}

\medskip\n
Further letting $L_K(w) = \sup\{n \ge 0; \,X_n(w) \in K\}, \;w \in W^0_K$, stand for the time of the last visit to $K$ of $w$, the law under $Q_K$ of $(X_n)_{0 \le n \le L_K}$ is supported on $\cT_K$, and for $A,B \in \cW_+$, $\tau \in \cT_K$, one has:
\begin{align}
&Q_K \big[(X_{-n})_{n \ge 0} \in A, \,(X)_{0 \le n \le L_K} = \tau, \,(X_{n + L_K})_{n \ge 0} \in B\big] = \label{1.26}
\\[0.5ex]
&P^K_{\tau(0)}[A] \,e_K\big(\tau(0)\big) \,P_{\tau(0)} [X_n = \tau(n), \,0 \le n \le N_\tau] \,e_K\big(\tau(N_\tau)\big) \,P^K_{\tau(N_\tau)}[B]\,. \nonumber
\\[4ex]
&\mbox{$\nu$ is invariant under the time reversal involution on $W^*$},\label{1.27}
\\
& \mbox{$w^* \rightarrow \check{w}^*$, where $\check{w}^* = \pi^*(\check{w})$, with $\pi^*(w) = w^*$ and $\check{w}(n) = w(-n)$, for $n \in \IZ$}\,. \nonumber
\\[4ex]
&\mbox{$\nu$ is invariant under the translations on $W^*$:}\label{1.28}
\\
& \mbox{$w^* \rightarrow w^* + x$, $x \in \IZ^d$, where $w^* + x = \pi^*(w + x)$, with $\pi^*(w) = w^*$}\,. \nonumber
\end{align}
\end{theorem}

\medskip
\begin{proof}
We begin with the proof of the existence and uniqueness of $\nu$. Since $W^* = \bigcup_{m \ge 0} \,W^*_{K_m}$, where $K_m \uparrow \IZ^d$, with $K_m$ finite, for $m \ge 0$, the uniqueness of $\nu$ satisfying (\ref{1.25}) is immediate. As for the existence of $\nu$, denote for $K \subset \subset \IZ^d$ with $\nu_K$ the finite measure supported on $W^*_K = \pi^*(W^0_K)$ in the right-hand side of (\ref{1.25}):
\begin{equation}\label{1.29}
\nu_K = \pi^* \circ Q_K\,.
\end{equation}

\n
The existence of $\nu$ will follow once we show that for $K \subset K^\prime \subset \subset \IZ^d$:
\begin{equation*}
1_{W^*_K} \,\nu_{K^\prime} = \nu_K\,.
\end{equation*}

\n
This in turn will follow once we prove that:
\begin{equation}\label{1.30}
(s_K \circ s_{K^\prime}^{-1}) \circ \big(1_{s_{K^\prime}(W^*_K)} \,Q_{K^\prime}\big) = Q_K\,,
\end{equation}

\medskip\n
where $s_{K^\prime}^{-1}$ denotes the restriction of $\pi^*$ to $W^0_{K^\prime}$. Indeed it simply suffices to take the image of both sides under $\pi^*$. We now write $s_{K^\prime}(W^*_K)$ as the at most countable partition into measurable sets:
\begin{equation}\label{1.31}
s_{K^\prime}(W^*_K) = \bigcup\limits_{\sigma \in \Sigma} \,W^0_{K^\prime,\sigma} \,,
\end{equation}

\n
where $\Sigma$ denotes the set of finite nearest-neighbor trajectories $\sigma = \big(\sigma(n)\big)_{0 \le n \le N_\sigma}$, with $\sigma(0) \in K^\prime$, $\sigma(n) \notin K$ for $n < N_\sigma$, and $\sigma(N_\sigma) \in K$, and
\begin{equation}\label{1.32}
W^0_{K^\prime,\sigma} = \{w \in W^0_{K^\prime}; \,X_n(w) = \sigma(n), \;\mbox{for} \;0 \le n \le N_\sigma\}\,.
\end{equation}
One then has the identity:
\begin{equation}\label{1.33}
s_K \circ s_{K^\prime}^{-1}(w) = w(\cdot + N_\sigma) = \theta_{N_\sigma}(w), \;\mbox{for} \;w \in W^0_{K^\prime,\sigma}\,.
\end{equation}

\medskip\n
As a result denoting with $Q$ the left-hand side of (\ref{1.30}), we find that
\begin{equation}\label{1.34}
Q = \dsl_{\sigma \in \Sigma} \,\theta_{N_\sigma} \circ \big(1_{W^0_{K^\prime,\sigma}} Q_{K^\prime}\big)\,.
\end{equation}

\medskip\n
Thus given an arbitrary collection $A_i$, $i \in \IZ$, of subsets of $\IZ^d$, we see that
\begin{equation}\label{1.35}
\begin{array}{l}
Q[X_i \in A_i, i \in \IZ] = \dsl_{\sigma \in \Sigma} \,Q_{K^\prime} [X_{i+N_\sigma} \in A_i, i \in \IZ, X_n = \sigma(n), 0 \le n \le N_\sigma]
\\[1ex]
= \dsl_{\sigma \in \Sigma} \,Q_{K^\prime} [X_i \in A_{i - N_\sigma}, i \in \IZ, X_n = \sigma(n), 0 \le n \le N_\sigma]
\\
\\[-1ex]
\hspace{-3ex} \stackrel{(\ref{1.6}), (\ref{1.24})}{=} \dsl_{\sigma \in \Sigma} \;\dsl_{x \in {\rm Supp}(e_{K^\prime})} \,P_x^{K^\prime} [X_m \in A_{-m - N_{\sigma}}, m \ge 0] \,P_x[\wt{H}_{K^\prime} = \infty]
\\
\\[-1ex]
\qquad \qquad P_x[X_n \in A_{n - N_\sigma}, n \ge 0, X_n = \sigma(n), 0 \le n \le N_\sigma]
\\[2ex]
\hspace{-3.5ex} \stackrel{(\ref{1.23}),{\rm Markov}}{=} \dsl_{\sigma \in \Sigma} \;\dsl_{x \in {\rm Supp} (e_{K^\prime})} \,P_x[X_m \in A_{-m - N_\sigma}, m \ge 0, \wt{H}_{K^\prime} = \infty]
\\
\\[-1ex]
\qquad \qquad \;\;\,P_x[X_n = \sigma(n) \in A_{n-N_\sigma}, 0 \le n \le N_\sigma] \,P_{\sigma(N_\sigma)} [X_n \in A_n, n \ge 0]\,.
\end{array}
\end{equation}

\medskip\n
It follows from the reversibility of the walk that for $y \in K$:
\begin{equation}\label{1.36}
\begin{array}{l}
\dsl_{\sigma: \sigma(N_\sigma) = y} \;\dsl_{x \in {\rm Supp}(e_{K^\prime})} P_x[X_m \in A_{-m - N_\sigma}, m \ge 0, \wt{H}_{K^\prime} = \infty]
\\
\\[-1ex]
P_x[X_n = \sigma(n) \in A_{n-N_\sigma}, 0 \le n \le N_\sigma]  = \dsl_{x \in {\rm Supp}(e_{K^\prime})} \; \dsl_{\sigma: \sigma(N_\sigma) = y \atop \sigma(0) = x}
\\[2ex]
P_x[X_{m} \in A_{-m - N_\sigma}, m \ge 0, \wt{H}_{K^\prime} = \infty]  
\\[1ex]
P_y[X_n = \sigma(N_\sigma - n) \in A_{-n}, 0 \le n \le N_\sigma] \stackrel{\rm Markov}{=}
\\[2ex]
\dsl_{x \in {\rm Supp}(e_{K^\prime})} \dsl_{\sigma: \sigma(N_\sigma) = y \atop \sigma(0) = x} P_y\big[X_n = \sigma(N_\sigma - n) \in A_{-n}, 0 \le n \le N_\sigma,   
\\[-2ex]
\qquad \qquad \qquad  \qquad \quad \;\,X_n \in A_{-n}, n \ge N_\sigma, \wt{H}_{K^\prime}\circ\theta_{N_\sigma} = \infty\big]
\\
\\[-1ex]
= \dsl_{x \in {\rm Supp}(e_{K^\prime})} P_y[\wt{H}_K = \infty,  \mbox{the last visit to $K^\prime$ occurs at $x$, and}
\\[-1.5ex]
\qquad \qquad \qquad  \quad \mbox{$X_n \in A_{-n}$, for $n \ge 0]$}
\\
\\[-1ex]
= P_y [\wt{H}_K = \infty, \,X_n \in A_{-n}, \,n \ge 0]\,.
\end{array}
\end{equation}

\medskip\n
Inserting this identity in the last line of (\ref{1.35}) we find that:
\begin{equation}\label{1.37}
\begin{array}{l}
Q[X_i \in A_i, i \in \IZ] = \dsl_{y \in K} \,P_y[\wt{H}_K = \infty, \, X_n \in A_{-n}, n \ge 0] \,P_y [X_n \in A_n, n \ge 0]
\\[2ex]
= \dsl_{y \in {\rm Supp}(e_K)} P_y^K [X_n \in A_{-n}, n \ge 0] \,e_K(y)\,P_y[X_n \in A_n, n \ge 0] 
\\
\\[-2ex]
\stackrel{(\ref{1.24})}{=} Q_K [X_n \in A_n, n \in \IZ]\,.
\end{array}
\end{equation}

\medskip\n
This proves that (\ref{1.30}) holds and thus concludes the proof of the existence of $\nu$ satisfying (\ref{1.25}), which is automatically $\sigma$-finite.

\medskip
We now turn to the proof of (\ref{1.26}). We consider $\tau(n), 0 \le n \le N$, some finite sequence in $\IZ^d$. Observe that  $Q_K [(X_n)_{0 \le n \le L_K}  = \tau]$ vanishes unless $\tau$ is nearest neighbor and $\tau(N) \in K$. Moreover when this is the case it follows from the use of the Markov property at time $N$ that:
\begin{equation}\label{1.38}
\begin{array}{l}
Q_K [(X_n)_{0 \le n \le L_K} = \tau] = Q_K [X_n = \tau(n), 0 \le n \le N, \wt{H}_K \circ \theta_N = \infty] 
\\[1ex]
\underset{\rm Markov}{\stackrel{(\ref{1.24}), (\ref{1.6})}{=}}  e_K\big(\tau(0)\big) \,P_{\tau(0)} [X_n = \tau(n), \,0 \le n \le N] \,e_K\big(\tau(N)\big)\,.
\end{array}
\end{equation}

\n
This shows that the law of $(X_n)_{0 \le n \le L_K}$ under $Q_K$ is supported by $\cT_K$. Also repeating the argument which yielded (\ref{1.38}), we see that for $A,B \in \cW_+$, $\tau \in \cT_K$ the left hand side of (\ref{1.26}) equals, (writing $N$ in place of $N_\tau$ for simplicity):
\begin{equation*}
\begin{array}{l}
Q_K\big[(X_{-n})_{n \ge 0} \in A, \;X_n = \tau(n), \,0 \le n \le N, \;\theta_N^{-1}\big(\{\wt{H}_K = \infty\} \cap \{X_n \in B, n \ge 0\}\big)\big] \stackrel{(\ref{1.24})}{=}
\\[1ex]
P^K_{\tau(0)}[A] \,e_K\big(\tau(0)\big)\,P_{\tau(0)} \big[X_n = \tau(n), \,0 \le n \le N, \,\theta_N^{-1}\big(\{\wt{H}_K = \infty\} \cap \{X_n \in B, n \ge 0\}\big)\big] 
\\[1ex]
\underset{(\ref{1.6}),(\ref{1.23})}{\stackrel{\rm Markov}{=}} P^K_{\tau(0)} [A]\,e_K\big(\tau(0)\big)\,P_{\tau(0)} [X_n = \tau(n), \,0\le n \le N] \,e_K\big(\tau(N)\big)\,P^K_{\tau(N)}[B]\,,
\end{array}
\end{equation*}

\n
and this proves (\ref{1.26}).

\medskip
To prove (\ref{1.27}), observe that for $K \subset \subset \IZ^d$, $w^* \rightarrow \check{w}^*$ leaves $W^*_K$ invariant and $X_n \big(s_K (\check{w}^*)\big) = X_{L_K - n} \big(s_K(w^*)\big)$, for $n \in \IZ$, $w^* \in W^*_K$. Denoting $\check{\gamma}$ the image under $w^* \rightarrow \check{w}^*$ of a measure $\gamma$ on $W^*$, we find for $C \in \cW$
\begin{equation}\label{1.39}
\begin{array}{lcl}
s_K \circ(1_{W^*_K} \check{\nu}) (C) &  \hspace{-1.5ex} = & \hspace{-1.5ex}s_K \circ \big((1_{W^*_K} \nu)\check{~}\big)(C) = s_K \circ  (1_{W^*_K} \nu)\big((X_{L_{K - \point}}) \in C\big)
\\[1ex]
& \hspace{-1.5ex} \stackrel{(\ref{1.25})}{=} &  \hspace{-1.5ex} Q_K \big((X_{L_{K - \point}}) \in C\big)\,.
\end{array}
\end{equation}

\medskip\n
Hence with (\ref{1.26}), $A,B \in \cW_+$, $\tau \in \cT_K$, and $C$ denoting the event in the probability in the first line of (\ref{1.26}), and writing $N$ in place of $N_\tau$ for simplicity we find that
\begin{equation}\label{1.40}
\begin{array}{l}
s_K \circ (1_{W^*_K} \,\check{\nu}) (C) = Q_K \big[(X_{-n})_{n \ge 0} \in B, 
\\
(X_n)_{0 \le n \le L_K} = \tau(N - \point), (X_{n + L_K})_{n \ge 0} \in A\big]  \stackrel{(\ref{1.26})}{=} P^K_{\tau(N)} [B] \,e_K \big(\tau(N)\big)
\\[1ex]
P_{\tau(N)} [(X_n)_{0 \le n \le N} = \tau(N-n)] \,e_K\big(\tau (0)\big)\,P^K_{\tau(0)}[A] \stackrel{\rm reversibility}{=} P^K_{\tau(0)} [A] \,e_K\big(\tau(0)\big)
\\[1.5ex]
P_{\tau(0)} [X_n = \tau(n), 0 \le n \le N] \,e_K \big(\tau(N)\big) \,P^K_{\tau(N)} [B] = Q_K(C) \stackrel{(\ref{1.25})}{=} 
\\[1.5ex]
s_K \circ (1_{W^*_K} \nu)(C)\,.
\end{array}
\end{equation}

\medskip\n
It now readily follows that $s_K \circ (1_{W^*_K} \check{\nu}) = s_K \circ (1_{W^*_K} \nu)$ and hence $1_{W^*_K} \check{\nu} = 1_{W^*_K} \nu$ for any $K \subset \subset \IZ^d$, whence $\check{\nu} = \nu$. This proves (\ref{1.27}).

\medskip\
Finally for the proof of (\ref{1.28}), we note that for $x \in \IZ^d$, $K \subset \subset \IZ^d$, $w^* \rightarrow w^* + x$ maps $W^*_K$ one-to-one onto $W^*_{K + x}$, and $s_K(w^* + x) = s_{K - x}(w^*) + x$, for $w^* \in W^*_{K - x}$. Denoting with $\gamma^x$ the image under $w^* \rightarrow w^* + x$ of a measure $\gamma$ on $W^*$, we see that for $C \in \cW$, we have
\begin{equation}\label{1.41}
\begin{array}{lcl}
s_K \circ(1_{W^*_K}\nu^x) (C) &  \hspace{-1.5ex} = & \hspace{-1.5ex} s_K \circ \big((1_{W^*_{K-x}} \nu)^x\big)(C) = s_{K-x}  \circ (1_{W^*_{K-x}} \nu)\big((X_n) + x \in C\big)
\\[1ex]
& \hspace{-1.5ex} \stackrel{(\ref{1.25})}{=} &  \hspace{-1.5ex} Q_{K-x} \big((X_n) + x \in C\big)\,.
\end{array}
\end{equation}

\n
Hence with $A,B \in \cW_+$, $y \in \IZ^d$ and $C$ denoting the event in the left-hand side of  (\ref{1.24}), where $x$ is replaced by $y$, we find that:
\begin{equation*}
\begin{array}{l}
s_K \circ (1_{W^*} \nu^x) (C) = Q_{K - x} \big[(X_{-n})_{n \ge 0} \in A-x,\,X_0 = y-x, \,(X_n)_{n \ge 0} \in B-x\big] \stackrel{(\ref{1.24})}{=}
\\[1ex]
P^{K - x}_{y-x}[A-x] \,e_{K - x}(y-x) \,P_{y - x}[B-x] = Q_K[C] \stackrel{(\ref{1.25})}{=} s_K \circ (1_{W^*_K} \nu) (C)\,,
\end{array}
\end{equation*}

\medskip\n
using (\ref{1.24}) and translation invariance in the third equality. This readily implies that $\nu^x = \nu$ and concludes the proof of Theorem \ref{theo1.1}.
\end{proof}

\begin{remark}\label{rem1.2} ~ \rm

\medskip\n
1) Let us say a few words on why the quotient space $W^*$ is better suited for our purpose than $W$. One can of course use the sections $s_K$, with $K$ growing along an increasing sequence of finite sets exhausting $\IZ^d$ to construct ``by patching'' a $\sigma$-finite measure on $(W,\cW)$ projecting down to $\nu$ under $\pi^*$. However there is no measure on $(W,\cW)$ invariant under translation of trajectories by constant vectors projecting down to $\nu$. Indeed if such a measure $\rho$ existed then for any $K \subset \subset \IZ^d$ we would have
\begin{equation*}
{\rm cap}(K) = \nu(W^*_K) = \rho(W_K) \ge \rho(X_0 \in K) = \rho(X_0 = 0) \,|K|\,,
\end{equation*}

\n
using translation invariance in the last equality. However capacity grows slower than volume when $K$ is of 
the form $B(0,L)$, with $L$ tending to infinity, cf.~below (\ref{3.24}). This would imply that $\rho(X_0 = 0) = 0$ 
and hence $\rho = 0$, due to translation invariance, thus leading to a contradiction. More obstructions can 
be brought to light, which make measures on $(W,\cW)$ projecting down to $\nu$ definitely less natural than $\nu$.

\bigskip\n
2) The expression in the right-hand member of (\ref{1.38}) when $K = B(0,L)$ coincides up to a normalization factor with the expression (3.13) of \cite{BenjSzni06} governing the limit law of certain properly recentered excursions of simple random walk on $(\IZ/N\IZ)^d$ to a box of side-length $2L$, see Theorem 3.1 of \cite{BenjSzni06}. This limiting result played a key role in the control of fluctuations of certain averages, cf.~(4.43) and Proposition 4.2. of \cite{BenjSzni06}. \hfill $\square$
\end{remark}

We will now endow the space $(\Omega, \cA)$, cf.~(\ref{1.16}), with a probability measure and thereby complete the construction of the basic model of interlacements. To this end we note that the infinite measure $\nu(dw^*) du$ on $W^* \times \IR_+$ gives finite mass to the sets $W^*_K \times [0,u]$, for $K \subset \subset \IZ^d$ and $u \ge 0$. We can thus construct on $(\Omega, \cA)$ the law $\P$ of a Poisson point measure with intensity $\nu(dw^*) du$. We denote with $\E [\cdot ]$ the corresponding expectation. The law $\IP$ is for instance characterized by the fact that, cf. \cite{Resn87}, p.~129,
\begin{equation}\label{1.42}
\begin{array}{l}
\IE\Big[\exp \Big\{ - \dis\int_{W^* \times \IR_+} \,f \,\o(dw^*, du)\Big\}\Big] = \exp\Big\{- \dis\int_{W^* \times \IR_+} (1 - e^{-f}) \,\nu(dw^*) du\Big\}\,,
\\
\\[-2ex]
\mbox{for any non-negative $\cW^* \otimes \cB(\IR_+)$-measurable function $f$}\,.
\end{array}
\end{equation}

\medskip\n
In a similar fashion we can also realize on $(M,\cM)$, cf.~(\ref{1.17}), the law of the Poisson point measure on $W_+ \times \IR_+$ with intensity $P_{e_K}(dw) du$, when $K \subset \subset \IZ^d$. We denote it with $\IP_K$ and write $\IE_K[\cdot]$ for the corresponding expectation. It is characterized by the fact that:
\begin{equation}\label{1.43}
\begin{array}{l}
\IE_K\Big[\exp \Big\{ - \dis\int_{W_+ \times \IR_+} \,f \,\mu(dw,du)\Big\}\Big] = \exp\Big\{- \dis\int_{W_+ \times \IR_+} (1 - e^{-f}) \,P_{e_K}(dw) du\Big\}\,,
\\
\\[-2ex]
\mbox{for any non-negative $\cW_+ \otimes \cB(\IR_+)$-measurable function $f$}\,.
\end{array}
\end{equation}

\medskip\n
We will now collect some straightforward properties of the laws $\IP$ and $\IP_K$. Given $\o = \sum_{i \ge 0} \delta_{(w^*_i,u_i)}$, we write 
\begin{equation}\label{1.44}
\begin{split}
\check{\o} & = \dsl_{i \ge 0} \delta_{(\check{w}_{i^*},u_i)} \in \Omega\,,
\\[1ex]
\tau_x \,\o & = \dsl_{i \ge 0} \delta_{(w^*_i - x,u_i)} \in \Omega, \;\mbox{for $x \in \IZ^d$} \,.
\end{split}
\end{equation}

\medskip\n
We also recall the notation from (\ref{1.18}), (\ref{1.19}).

\begin{proposition}\label{prop1.3} $(K \subset K^\prime \subset \subset \IZ^d)$
\begin{align}
&\mbox{$\IP_K$ is the law of $\mu_K$ under $\P$}\,. \label{1.45}
\\[1ex]
& \Theta_K \circ \IP_{K^\prime} = \IP_K\,.\label{1.46}
\\[1ex]
&\mbox{$\IP$ is invariant under $\o \rightarrow \check{\o}$, (time-reversal invariance)}\,.\label{1.47}
\\[1ex]
&\mbox{$\IP$ is invariant under $\tau_x$ for any $x \in \IZ^d$, (translation invariance)}\,.\label{1.48}
\end{align}
\end{proposition}

\begin{proof}
We begin with (\ref{1.45}), and note that $\mu_K$ due to (\ref{1.18}) is distributed as a Poisson point process on $W_+ \times \IR_+$ with intensity measure $\gamma(dw \,du)$ such that for $f$ as in (\ref{1.18})
\begin{equation}\label{1.49}
\begin{array}{l}
\dis\int_{W_+ \times \IR_+} f \,\gamma(dw, du) = \dis\int_{W^*_K \times \IR_+} f(s_K(w^*)_+, u) \;\nu(dw^*) du \stackrel{(\ref{1.25}), (\ref{1.24})}{=} 
\\
\\[-1ex]
\dis\int_{W_+ \times \IR_+} f(w,u)\,P_{e_K} (dw) du\,.
\end{array}
\end{equation}

\n
This shows that the law of $\mu_K$ coincides with $\IP_K$. Then (\ref{1.46}) immediately follows from (\ref{1.21}) i) and (\ref{1.45}), whereas (\ref{1.47}), (\ref{1.48}) respectively follow from (\ref{1.27}), (\ref{1.28}).
\end{proof}

\begin{remark}\label{rem1.4}
\rm The constructions we have made here in the case of simple random walk on $\IZ^d$, $d \ge 3$, can be straightforwardly generalized to the case of an infinite locally finite connected graph $\Gamma = (G,\cE)$ with vertex set $G$ and (undirected) edge set $\cE$, endowed with positive weights
\begin{equation}\label{1.50}
\lambda(e) > 0, \;e \in \cE\,,
\end{equation}

\n
so that the corresponding nearest neighbor walk on $G$ with transition probability
\begin{equation}\label{1.51}
\begin{split}
p_{x,y} & = \dis\frac{\lambda(\{x,y\})}{\sum\limits_{z: \{x,z\} \in \cE} \lambda(\{x,z\})}, \;\mbox{if} \;\{x,y\} \in \cE\,,
\\[1ex]
& = 0, \;\mbox{otherwise}\,.
\end{split}
\end{equation}
is transient. This walk is reversible with respect to the measure
\begin{equation}
\lambda_x = \dsl_{y: \{x,y\} \in \cE} \lambda(\{x,y\}), \;x \in G.
\end{equation}

\n
In this set-up some of our definitions need to be modified. For instance if $P_x$ stands for the law of the walk starting from $x \in G$, one divides the right-hand side of (\ref{1.5}) by $\lambda_y$, and multiplies the right-hand side of (\ref{1.6}) by $\lambda_x$, cf.~\cite{Telc06}.

\medskip
The results we stated in Theorem \ref{theo1.1} and Proposition \ref{prop1.3}, except for (\ref{1.28}), (\ref{1.48}), which explicitly refer to the additive structure of $\IZ^d$ can easily be extended to this set-up. We refrain from doing this here since the main results of this article will pertain to percolation properties of the vacant set, which we introduce below, and rely on the structure of $\IZ^d$. \hfill $\square$
\end{remark}

We can now define for $\o \in \Omega$, the {\it interlacement at level $u$}, as the subset of $\IZ^d$:
\begin{equation}\label{1.53}
\begin{array}{lcl}
\cI^u(\o) &\hspace{-1.5ex} = &\hspace{-1.5ex}\bigcup\limits_{u_i \le u} {\rm range}(w_i^*), \;\mbox{if} \;\o = \dsl_{i \ge 0} \delta_{(w_i^*,u_i)} \in \Omega, \;u \ge 0\,,
\\[1ex]
&\hspace{-1.5ex} \stackrel{(\ref{1.20})}{=} &\hspace{-1.5ex} \bigcup\limits_{K \subset \subset \IZ^d} \;\bigcup\limits_{w \in {\rm Supp}\, \mu_{K,u}(\o)} w(\IN)\,,
\end{array}
\end{equation}

\n
where for $w^* \in W^*$, range$(w^*) = w(\IZ)$, for any $w \in W$ with $\pi^*(w) = w^*$. Note that in view of (\ref{1.18}), (\ref{1.20}), the following identity holds:
\begin{equation}\label{1.54}
\cI^u(\o) \cap K = \bigcup\limits_{w \in {\rm Supp}\, \mu_{K^\prime,u}(\o)} w(\IN) \cap K, \;\mbox{for any} \;K \subset K^\prime \subset \subset \IZ^d\,.
\end{equation}

\medskip\n
The {\it vacant set at level $u$} is then defined as 
\begin{equation}\label{1.55}
\cV^u(\o) = \IZ^d \backslash \cI^u(\o), \;\o \in \Omega, \,u \ge 0\,.
\end{equation}

\medskip\n
Obviously with (\ref{1.53}), (\ref{1.55}), $\cI^u(\o)$ increases with $u$, whereas $\cV^u(\o)$ decreases with $u$. In the next proposition we collect some simple properties of these random subsets. Given $K, \wt{K} \subset \subset \IZ^d$, we say that $\wt{K}$ separates $K$ from infinity when any nearest neighbor path starting in $K$ and tending to infinity enters $\wt{K}$.

\begin{proposition}\label{prop1.5} $(u \ge 0, K, \wt{K}\ss \IZ^d)$
\begin{align}
&\cI^u(\o) \cap K \not= \emptyset \Longleftrightarrow \mu_{K,u}(\o) \not= 0, \;\mbox{for}\; \o \in \Omega\,, \label{1.56}
\\
&\mbox{and $\cI^u, \cV^u$ depend measurably on $\o$}\,. \nonumber
\\[2ex]
&\IP[K \subseteq \cV^u] = \exp\{ - u \,{\rm cap} (K)\}\,. \label{1.57}
\\[2ex]
&\IP[x \in \cV^u] = \exp\Big\{- \dis\frac{u}{g(0)}\Big\}, \;\mbox{for $x \in \IZ^d$} \,.\label{1.58}
\\[1ex]
&\IP[\{x,y\} \subseteq \cV^u] = \exp\Big\{ - \dis\frac{2u}{g(0) + g(y-x)}\Big\}, \;\mbox{for $x,y \in \IZ^d$}\,.\label{1.59}
\end{align}

\medskip\n
If $\wt{K}$ separates $K$ from infinity then the following inclusion holds
\begin{equation}\label{1.60a}
\{\cV^u \supseteq \wt{K}\} \subseteq \{\cV^u \supseteq K\}, \;\mbox{(screening effect)} \,.
\end{equation}
\end{proposition}

\medskip
\begin{proof}
The claim (\ref{1.56}) immediately follows from (\ref{1.54}) when $K^\prime =  K$, and (\ref{1.18}), (\ref{1.20}). The measurability of the sets $\cI^u$, $\cV^u$ (understood as the measurability of the maps $1\{x \in \cI^u\}$ and $1\{x \in \cV^u\}$ for all $x \in \IZ^d$) is a direct consequence of the above statement. With (\ref{1.56}), we thus see that
\begin{equation}\label{1.61}
\IP[\cV^u \supseteq K] = \IP[\mu_{K,u} = 0] \stackrel{(\ref{1.20}), (\ref{1.43})}{=} \exp \{- u \,P_{e_K} (W_+)\} \stackrel{(\ref{1.7})}{=} \exp\{- u \,{\rm cap} (K)\}\,,
\end{equation}

\medskip\n
and this proves (\ref{1.57}). As a result of (\ref{1.6}) or (\ref{1.8}) one finds that
\begin{equation}\label{1.62}
{\rm cap}(\{x\}) = g(0)^{-1}, \;\mbox{for $x \in \IZ^d$}\,,
\end{equation}

\medskip\n
and (\ref{1.58}) follows from (\ref{1.57}). As for (\ref{1.59}), we can assume without loss of generality that $x \not= y$, and note that for suitable $\rho_x,\rho_y > 0$, one has
\begin{equation}\label{1.63}
e_{\{x,y\}} = \rho_x \delta_x + \rho_y \delta_y, \quad {\rm cap}(\{x,y\}) = \rho_x + \rho_y\,,
\end{equation}

\medskip\n
so that with (\ref{1.8}) one finds:
\begin{equation*}
g(z,x) \,\rho_x + g(z,y) \,\rho_y = 1, \;\mbox{for $z = x,y$}\,.
\end{equation*}
Solving this system of equations we see that $\rho_x = \rho_y = \big(g(0) + g(y - x)\big)^{-1}$, and hence
\begin{equation}\label{1.64}
{\rm cap}(\{x,y\}) = \dis\frac{2}{g(0) + g(y-x)}, \;\mbox{for $x,y \in \IZ^d$}\,.
\end{equation}

\medskip\n
The claim (\ref{1.59}) now follows from (\ref{1.57}). 

\medskip
Finally note that when $\wt{K}$ separates $K$ from infinity, $w^* \in W^*_K \Longrightarrow w^* \in W^*_{\wt{K}}$, and with (\ref{1.56}) we see that $\cI^u(\o) \cap K \not= \emptyset \Longrightarrow \cI^u(\o) \cap \wt{K} \not= \emptyset$, whence (\ref{1.60a}). 
\end{proof}

\begin{remark}\label{rem1.6}  ~ \rm

\medskip\n
1) Using estimates on the  capacity of a large cube, cf.~for instance (2.4) in Lemma 2.2 of \cite{BoltDeus93} and \cite{Spit01}, p.~341, one sees that for $u \ge 0$,
\begin{equation}\label{1.65a}
\IP[\cV^u \supseteq B(0,L)] = \exp\big\{ - cu\,L^{d-2}\big(1 + o(1)\big)\big\}, \;\mbox{as $L \r \infty$}\,.
\end{equation}

\medskip\n
In particular there is no general exponential decay with $|A|$ of $\IP[\cV^u \supseteq A]$. This feature is drastically different from what happens for Bernoulli site percolation, see \cite{Grim99}. It creates very serious difficulties when trying to prove that for large $u$, $\cV^u$ does not percolate, see Section 3. Also (\ref{1.65a}) can be compared with (4.58), (4.62) of Benjamini-Sznitman \cite{BenjSzni06}, in the case of the vacant set left by simple random walk on $(\IZ / N\IZ)^d$ up to time $[u N^d]$.

\medskip
Incidentally in spite of the fact that $\cV^u$ displays a tendency to contain bigger boxes than Bernoulli site percolation, no matter how small $u > 0$, the law $Q_u$ of $1\{x \in \cV^u\}$, $x \in \IZ^d$, on $\{0,1\}^{\IZ^d}$, does not stochastically dominate Bernoulli site percolation with parameter close to $1$. Indeed the complement $\cI^u$ of $\cV^u$ always percolates.

\bigskip\n
2) As a direct consequence of (\ref{1.57}), and the inequality ${\rm cap}(K \cup K^\prime) \le {\rm cap}(K) + {\rm cap}(K^\prime)$, we see that 
\begin{equation}\label{1.66a}
\IP[K \cup K^\prime \subseteq \cV^u] \ge \IP [K \subseteq \cV^u]\, \IP [K^\prime \subseteq \cV^u], \; \mbox{for}\; K, K^\prime \subset \subset \IZ^d, u \ge 0\,,
\end{equation}

\n
i.e. the events $\{K \subseteq \cV^u\}$, $\{K^\prime \subseteq \cV^u\}$ are positively correlated. However we do not know whether the FKG inequality holds under the law $Q_u$ mentioned in 1).

\bigskip\n
3) As a direct consequence of (\ref{1.54}) and (\ref{1.26}), for $K \subset \subset \IZ^d$ we can visualize $\cI^u \cap K$ as the trace left on $K$ by a Poisson point process of finite trajectories belonging to the space $\cT_K$ of (\ref{1.22}). More precisely for any $u \ge 0$,
\begin{equation}\label{1.67}
\begin{array}{l}
\mbox{$\cI^u \cap K$ has the same distribution under $\IP$ as the trace on $K$ of a}
\\[0.5ex]
\mbox{Poisson point process of trajectories on $\cT_K$ with intensity measure}
\\[0.5ex]
\rho^u_K(\tau) = u \,e_K\, \tau(0)\,P_{\tau(0)} [X_n = \tau(n), 0 \le n \le N_\tau] \,e_K \big(\tau(N_\tau)\big), \;\mbox{for} \;\tau \in \cT_K\,.
\end{array}
\end{equation}

\bigskip\n
This has  a very similar flavor to some of the results in Section 3 and 4 of \cite{BenjSzni06}.

\bigskip\n
4) With standard estimates on the behavior of $g(\cdot)$ at infinity, cf.~\cite{Lawl91}, p.~31, one sees that for any $u \ge 0$,
\begin{equation}\label{1.67a}
\begin{split}
{\rm cov}_\IP \big(1_{\{x \in \cV^u\}}, 1_{\{y \in \cV^u\}})  & \sim \dis\frac{2u}{g(0)^2} \;g(y-x) \,e^{-\frac{2 u}{g(0)}}
\\[1ex]
& \sim \dis\frac{cu}{|y - x|^{d-2}} \;e^{-cu}, \;\mbox{as $|y - x| \r \infty$}\,,
\end{split}
\end{equation}

\medskip\n
where ${\rm cov}_\IP$ denotes the covariance under $\IP$. This displays the presence of long range correlations in the random set $\cV^u$. 

\bigskip\n
5) Formulas (\ref{1.58}), (\ref{1.59}) are in essence (2.26)  and (3.6) in Brummelhuis-Hilhorst \cite{BrumHilh91}, concerning the large $N$ behavior of the probability that one or two given points in $(\IZ/N\IZ)^d$ are not visited by simple random walk up to time $t = [uN^d]$. The prefactors present in formulas (2.26), (3.6) of \cite{BrumHilh91} stem from the fact that the walk under consideration starts at the origin and not with the uniform distribution as in \cite{BenjSzni06}. For a similar interpretation of (\ref{1.58}) see also Aldous-Fill \cite{AldoFill99} Chapter 3, Proposition 20, and Chapter 13, Proposition 8. One can also compare (\ref{1.57}) with Propositions 20 and 37 in Chapter 3 of \cite{AldoFill99}.

\hfill $\square$
\end{remark}

 \section{A zero-one law and an exponential bound}
 \setcounter{equation}{0}
 
 In this section we exploit the translation invariance of the basic model in a more substantial way. We prove that the probability that $\cV^u$, the vacant set at level $u$, contains an infinite connected component is either zero or one. This zero-one law comes as a consequence of the ergodicity of the law of $\cV^u$, cf.~Theorem \ref{theo2.1}. We also show in Corollary \ref{cor2.3} that with probability one $\cI^u$ is connected. In Theorem \ref{theo2.4} we prove an exponential bound on the probability that $\cI^u$ contains a given subset of an $m$-dimensional discrete subspace of $\IZ^d$, with $m \le d-3$. This result has a similar flavor to Theorem 2.1 of \cite{BenjSzni06}, or Theorem 1.2 of \cite{DembSzni08}, but has a more algebraic proof due to the nature of our basic model. Combined with a Peierls-type argument, cf.~Remark \ref{rem2.5}, it can be used to show that when $d$ is large enough, $\cV^u$ percolates when $u$ is chosen sufficiently small. In Section 4 we will present a more powerful method proving such a result as soon as $d \ge 7$. We begin with some notation.
 
 \medskip
 We denote with $Q_u$, the law on $\{0,1\}^{\IZ^d}$ of $(1 \{x \in \cV^u\})_{x \in \IZ^d}$, for $u \ge 0$. We write $Y_x, x \in \IZ^d$, for the canonical coordinates on $\{0,1\}^{\IZ^d}$, $\cY$ for the canonical $\sigma$-algebra, and $t_x, x \in \IZ^d$, for the canonical shift. We also consider for $u \ge 0$ the event
 \begin{equation}\label{2.1}
 \mbox{Perc$(u) = \{\o \in \Omega; \;\cV^u(\o)$ contains an infinite connected component\}},   
 \end{equation}
as well as
\begin{equation}\label{2.2}
 \mbox{$\eta(u) = \IP[0$ belongs to an infinite connected component of $\cV^u]$}\,. 
 \end{equation}
 
\medskip \n
 The first main result of this section is:
 \begin{theorem}\label{theo2.1} $(d \ge 3)$
\begin{align}
&\mbox{For any $u \ge 0$, $(t_x)_{x \in \IZ^d}$ is a measure preserving flow on $(\{0,1\}^{\IZ^d}, \cY, Q_u)$} \label{2.3}
\\
&\mbox{which is ergodic}\,. \nonumber
\\[2ex]
&\mbox{For any $u \ge 0$, $\IP[{\rm Perc}(u)] = 0$ or $1$}\,. \label{2.4}
\end{align}
 \end{theorem}
 
\begin{proof}
We begin with the proof of (\ref{2.3}). We denote with $\psi_u: \Omega \r \{0,1\}^{\IZ^d}$, the map $\psi_u(\o) = \big(1\{x \in \cV^u(\o)\}\big)_{x \in \IZ^d}$, so that $Q_u = \psi_u \circ\IP$. Note that with (\ref{1.44}), (\ref{1.53}), (\ref{1.55}), one has
\begin{equation}\label{2.5}
t_x \circ \psi_u = \psi_u \circ \tau_x, \;\mbox{for $x \in \IZ^d$} \,.
\end{equation}

\medskip\n
Since $\IP$ is invariant under $(\tau_x)$, cf.~(\ref{1.48}), it follows that $Q_u$ is invariant under $(t_x)$. To prove the ergodicity of $(t_x)$, we argue as follows. We consider $u \ge 0$, and note that the claim will follow once we show that for any $K \subset \subset \IZ^d$, and any $[0,1]$-valued $\sigma(Y_z,z \in K)$-measurable function $f$ on $\{0,1\}^{\IZ^d}$, one has
\begin{equation}\label{2.6}
\lim\limits_{|x|\r\infty} E^{Q_u} [f \,f \circ t_x] = E^{Q_u}[f]^2\,.
\end{equation}

\medskip\n
Indeed the indicator function of any $A \in \cY$ invariant under $(t_x)_{x \in \IZ^d}$ can be approximated in $L^1(Q_u)$ by functions $f$ as above. With (\ref{2.6}) one classically deduces that necessarily $Q_u(A) = Q_u(A)^2$, whence $Q_u(A) \in \{0,1\}$. In view of (\ref{1.54}), with $K = K^\prime$, and (\ref{2.5}), the claim (\ref{2.6}) will follow once we show that for any $K \subset \subset \IZ^d$:
\begin{equation}\label{2.7}
\lim\limits_{|x|\r \infty} \E[F(\mu_{K,u}) \;F(\mu_{K,u}) \circ \tau_x] = \IE[F(\mu_{K,u})]^2\,, 
\end{equation}

\n
for any $[0,1]$-valued measurable function $F$ on the set of finite point-measures on $W_+$, endowed with its canonical $\sigma$-field. With (\ref{1.20}), (\ref{1.44}), we can find $G$ (depending on $x$), with similar properties as $F$, such that the expectation in the left-hand side of (\ref{2.7}) equals $\IE[F(\mu_{K,u}) \,G(\mu_{K + x,u})]$.

\medskip
From now on we assume $|x|$ large enough so that $K \cap (K + x) = \phi$. To control the above expectation we are going to express both $\mu_{K,u}$ and $\mu_{K + x, u}$ in terms of $\mu_{K \cup (K+x),u}$, with the help of (\ref{1.21}) i), and extract the desired asymptotic independence. We will recurrently use this type of decomposition in what follows. Namely with $V = K \cup (K+x)$ we write:
\begin{equation}\label{2.8}
\begin{split}
\mu_{V,u} & = \mu_{1,1} + \mu_{1,2} +  \mu_{2,1} + \mu_{2,2}, \;\;\mbox{where}
\\[1ex]
\mu_{1,1}(dw) & = 1\{X_0 \in K, \,H_{x+K} = \infty\} \,\mu_{V,u} (dw)\,,
\\[1ex]
\mu_{1,2}(dw) & = 1 \{X_0 \in K, \,H_{x+K} < \infty\} \,\mu_{V,u}(dw)\,,
\end{split}
\end{equation}

\medskip\n
and similar formulas for $\mu_{2,2}$ and $\mu_{2,1}$ with the role of $K$ and $K + x$ exchanged. It follows from (\ref{1.20}), (\ref{1.45}) that
\begin{equation}\label{2.9}
\mbox{$\mu_{i,j}$, $1 \le i, j \le 2$, are independent Poisson point processes on $W_+$}\,,
\end{equation}

\medskip\n
and their respective intensity measures are:
\begin{equation}\label{2.10}
\begin{array}{ll}
\gamma_{1,1} = u 1\{X_0 \in K, H_{K+x} = \infty\} \,P_{e_V}, & \gamma_{1,2} = u 1\{X_0 \in K, H_{K+x} < \infty\} \,P_{e_V},
\\[1ex]
\gamma_{2,1} = u 1\{X_0 \in K + x, H_K < \infty\} \,P_{e_V}, & \gamma_{2,2} = u 1\{X_0 \in K + x, H_K = \infty\} \,P_{e_V}\,.
\end{array}
\end{equation}

\n
As a consequence of (\ref{1.20}), (\ref{1.21}) i), we see that
\begin{equation}\label{2.11}
\begin{split}
\mu_{K,u} & = \mu_{1,1} + \mu_{1,2} + \ov{\mu}^K_{2,1} \;,
\\[1ex]
\mu_{K+x,u} & = \mu_{2,2} + \mu_{2,1} + \ov{\mu}_{1,2}^{K+x}\,,
\end{split}
\end{equation}

\medskip\n
where given $U \subset \subset \IZ^d$, and $\mu(dw) = \sum_{0 \le i \le N} \delta_{w_i}$ a finite point measure on $W_+$, $\ov{\mu}^U(dw) = \sum_{0 \le i \le N} \delta_{\theta_{H_U}(w_i)} 1\{H_U(w_i) < \infty\}$, and we have used in (\ref{2.11}) the fact that $\ov{\mu}^K_{2,2} = 0$, and $\ov{\mu}_{1,1}^{K+x} = 0$. Therefore introducing auxiliary independent Poisson point processes $\mu^\prime_{1,2}$, $\mu^\prime_{2,1}$, independent of $\mu_{i,j}$, $1 \le i,j \le 2$, with the same distribution as $\mu_{1,2}, \mu_{2,1}$ respectively, we find that
\begin{equation}\label{2.12}
\mu^\prime_{K,u} \stackrel{\rm def}{=} \mu_{1,1} + \mu_{1,2} + \ov{\mu^\prime}_{2,1}\,{\hspace{-2.5ex}^{K}}\;, \;\mu^\prime_{K + x,u} \stackrel{\rm def}{=} \mu_{2,2} + \mu_{2,1} + \ov{\mu^\prime}_{1,2}\,{\hspace{-2.5ex}^{K+x}} \,,
\end{equation}

\n
are independent point processes respectively distributed as $\mu_{K,u}$ and $\mu_{K + x,u}$. With the same notation as in (\ref{1.67a}) we find that 
\begin{equation}\label{2.13}
\begin{array}{l}
\big|{\rm cov}_\IP\big(F(\mu_{K,u}), \,G(\mu_{K+x,u})\big)\big| = 
\\
\big|\IE[F(\mu_{K,u}) \,G(\mu_{K+x,u}) - F(\mu^\prime_{K,u}) 
\,G(\mu^\prime_{K+x,u})] \,|\,\stackrel{(\ref{2.11}), (\ref{2.12})}{\le}
\\
\mbox{$\IP[\mu_{1,2}$ or $\mu_{2,1}$ or $\mu^\prime_{1,2}$ or $\mu^\prime_{2,1}$ is different from $0] \stackrel{(\ref{2.9}),(\ref{2.10})}{\le}$}
\\[2ex]
2(1 - \exp\{ - \gamma_{1,2}(W_+)\}) + 2(1 - \exp\{-\gamma_{2,1}(W_+)\}) \le
\\[2ex]
2 u (P_{e_V} [X_0 \in K, H_{K+x} < \infty] + P_{e_V}[X_0 \in K + x, H_K < \infty])\,,
\end{array}
\end{equation}

\medskip\n
where in the last step we have used the inequality $1 - e^{-v} \le v$, for $v \ge 0$, in addition to (\ref{2.10}). Observe now that
\begin{equation}\label{2.14}
\begin{array}{l}
P_{e_V} [X_0 \in K, H_{K+x} < \infty] = \dsl_{z \in K} e_V(z) \,P_z [H_{K+x} < \infty] \stackrel{(\ref{1.8})}{=}
\\[1ex]
\dsl_{z \in K, y \in K+x} e_V (z) \,g(z,y) \,e_{K+x}(y) \stackrel{(\ref{1.6})}{\le} \dsl_{z \in K, y \in K+x} e_K(z) \,g(z,y) \,e_{K+x}(y) \le  
\\
\\[-1ex]
c\;\dis\frac{{\rm cap}(K)^2}{d(K,K+x)^{d-2}} \,,
\end{array}
\end{equation}

\medskip\n
with the notation introduced above (\ref{1.1}), as well as standard bounds on the Green function, cf.~\cite{Lawl91}, p.~31, and translation invariance. A similar bound holds for $P_{e_V} [X_0 \in K+x, H_K < \infty]$, and with (\ref{2.13}) we see that for $u \ge 0$, $K \subset \subset \IZ^d$, $x \in \IZ^d$, $F,G$-measurable functions on the set of finite point measures on $W_+$ with values in $[0,1]$,
\begin{equation}\label{2.15}
\big|{\rm cov}_\IP \big(F(\mu_{K,u}), \,G(\mu_{K+x,u})\big)\big| \le c\,u \;\dis\frac{{\rm cap}(K)^2}{d(K,K+x)^{d-2}} \;.
\end{equation}

\n
This implies (\ref{2.7}) and thus concludes the proof of (\ref{2.3}). As for (\ref{2.4}), note that Perc$(u) = \psi_u^{-1}(A)$, where $A \in \cY$ stands for the invariant event consisting of configurations in $\{0,1\}^{\IZ^d}$ such that there is an infinite connected component in the subset of $\IZ^d$ where the configuration takes the value $1$. It now follows from (\ref{2.3}) that $Q_u(A) = \IP[{\rm Perc}(u)]$ is either $0$ or $1$. This proves (\ref{2.4}).
\end{proof}

\begin{remark}\label{rem2.2} ~\rm

\medskip\n
1) Note that (\ref{2.15}) has a similar flavor to (\ref{1.67a}), which mirrors the long range dependence built into the basic model. Taming this effect will bring some serious difficulties in \linebreak Section 3. 

\bigskip\n
2) One can characterize $Q_u$ as the unique probability on $(\{0,1\}^{\IZ^d}, \cY)$ such that
\begin{equation}\label{2.16}
Q_u(Y_z = 1, \;\mbox{for}\;z \in K) = \exp\{- u \;{\rm cap}(K)\}, \;\mbox{for any}\; K \subset \subset \IZ^d\,.
\end{equation}

\n
Indeed the collection of events which appear in (\ref{2.16}) is stable under finite intersection and generates $\cY$. In a slightly more constructive fashion, we see with a classical inclusion exclusion argument that for any disjoint finite subsets $K, K^\prime$ of $\IZ^d$, one has 
\begin{equation}\label{2.17}
\begin{array}{l}
Q_u [Y_z = 1, \;\mbox{for} \;z \in K, \,Y_z = 0, \;\mbox{for} \;z \in K^\prime] =
\\[1ex]
E^{Q_u} \Big[\prod\limits_{z \in K} Y_z \prod\limits_{z \in K^\prime} \,(1 - Y_z) \Big] = \dsl_{A \subseteq K^\prime} (-1)^{|A|} \exp\{ - u \;{\rm cap}(K \cup A)\}\,.
\end{array}
\end{equation}

\medskip\n
3) The present work does not address the question of whether there is a unique infinite connected component in $\cV^u$ when it percolates and $u$ is positive. The answer to this question is affirmative, as proved in \cite{Teix08}. The classical results of Burton-Keane \cite{BurtKean89},
see also \cite{HaggJona06}, p.~326,~332, implying such a uniqueness
do not apply because, as one easily sees, $Q_u$ fails to fulfill the so-called finite energy condition:
\begin{equation*}
\mbox{$0 < Q_u( Y_x = 1 | Y_z, z \not= x ) < 1$, $Q_u$-a.s., for all $x \in \IZ^d$}\,.
\end{equation*}
Loosely speaking the problem stems from the fact that the set of sites $w$, where $Y_w$ takes
the value $0$, has no bounded component, and on some configurations turning a value $0$ into
a value $1$, say at the origin, can lead to a forbidden configuration,
(see also (\ref{1.67})). In Corollary \ref{cor2.3} we are able to adapt the argument of Burton-Keane 
in the case of $\cI^u$, and prove that with probability one $\cI^u$ is connected. In the case of $\cV^u$ 
the construction of so-called trifurcations is more delicate, and can be found in \cite{Teix08}.

\bigskip\n
4) Denote with $\IE_d = \big\{\{x,y\}; x, y$ in $\IZ^d$ with $|x-y| = 1\big\}$, the collection of nearest neighbor edges on $\IZ^d$. Given $\o \in \Omega$ and $u \ge 0$, one can consider the subset $\wt{\cI}^u(\o)$ of $\IE_d$ consisting of the edges which are traversed by at least one of the trajectories at level $u$ entering $\o$:
\begin{equation}\label{2.18new}
\begin{split}
\wt{\cI}^u(\o) = \big\{ & e \in \IE_d;\; \mbox{for some $i \ge 0$, with $u_i \le u$ and $n \in \IZ$,}
\\
& e = \{w_i(n), w_i(n+1)\}\big\}, \;\mbox{if $\o = \sum_{i \ge 0} \delta_{(w_i^*,u_i)} \in \Omega$}\,,
\end{split}
\end{equation}

\n
and $w_i$ is any element of $W$ with $\pi^*(w_i) = w_i^*$.

\medskip
Connected components of $\IZ^d$ induced by $\wt{\cI}^u(\o)$ are either singletons in $\cI^u(\o)^c$ or infinite components partitioning $\cI^u(\o)$. Denoting with $\wt{\psi}_u$: $\Omega \rightarrow \{0,1\}^{\IE_d}$ the map $\wt{\psi}_u(\o) = ( 1\{e \in \wt{\cI}^u(\o)\})_{e \in \IE_d}$, one can consider the image $\wt{Q}_u$ on $(\{0,1\}^{\IE_d},\wt{\cY})$ of $\IP$ under $\wt{\psi}_u$, where $\wt{\cY}$ stands for the canonical $\sigma$-algebra on $\{0,1\}^{\IE_d}$. With $\wt{t}_x$, $x \in \IZ^d$, the canonical shift on $\{0,1\}^{\IE_d}$, one finds exactly as in (\ref{2.5}) that $\wt{t}_x \circ \wt{\psi}_u = \wt{\psi}_u \circ \tau_x$, for $x \in \IZ^d$. The same proof as in (\ref{2.3}), see in particular (\ref{2.7}), now yields that
\begin{equation}\label{2.19new}
\begin{array}{l}
\mbox{for any $u \ge 0$, $(\wt{t}_x)_{x \in \IZ^d}$ is a measure preserving flow on $(\{0,1\}^{\IE_d}, \wt{\cY}, \wt{Q}_u)$}
\\[0.5ex]
\mbox{which is ergodic.}
\end{array}
\end{equation}
\hfill $\square$
\end{remark}

The first statement below is an immediate consequence of Theorem \ref{theo2.1} and (\ref{2.2}).
\medskip\n

\begin{corollary}\label{cor2.3} $(d \ge 3)$

\medskip
For $u \ge 0$, one has the equivalences
\begin{equation}\label{2.16a}
\begin{array}{ll}
{\rm i)} & \IP[{\rm Perc}(u)] = 1 \Longleftrightarrow \eta(u) > 0 \,,
\\[1ex]
{\rm ii)} & \IP[{\rm Perc}(u)] = 0 \Longleftrightarrow  \eta(u) = 0 \,.
\end{array}
\end{equation}
\begin{equation}\label{2.16b}
\mbox{For $u > 0$, $\IP$-a.s., $\cI^u$ is an infinite connected subset of $\IZ^d$}\,.
\end{equation}

\end{corollary}

\begin{proof}
We begin with (\ref{2.16a}). One simply needs to observe that 
\begin{equation*}
\eta(u) \le \IP[{\rm Perc}(u)] \le \dsl_{x \in \IZ^d}  \;\mbox{$\IP[x$ belongs to an infinite connected component of $\cV^u$}]\,,
\end{equation*}

\n
and in view of (\ref{1.48}) all summands in the right-hand side equal $\eta(u)$. The claim (\ref{2.16a}) now follows from the zero-one law (\ref{2.4}).

We now turn to the proof of (\ref{2.16b}), which is an adaptation of the argument of Burton-Keane \cite{BurtKean89}. The consideration of $\wt{\cI}^u$, cf.~Remark \ref{rem2.2} 4) will be helpful, see in particular the observation below (\ref{2.18new}).  With the ergodicity property (\ref{2.19new}), it follows that the total number $N_u$ of infinite connected components determined by $\wt{\cI}^u$ is $\IP$-a.s. equal to a positive, possibly infinite, constant. With the observation below (\ref{2.18new}) our claim (\ref{2.16b}) will follow once we show that this constant equals 1. The first step, see also \cite{NewmSchu81}, is to argue that
\begin{equation}\label{2.20a}
\mbox{for $2 \le k < \infty, \;\IP[N_u = k] = 0$}\,.
\end{equation}

\medskip\n
Assume instead that for some $2 \le k < \infty$, $\IP[N_u = k] = 1$. Then we can find $K = B(0,L)$ such that $\IP[A] > 0$, where $A$ denotes the event $\{N_u = k$ and $K$ intersects two distinct infinite components determined by $\wt{\cI}^u(\o)\}$. Note that under $\IP$
\begin{equation}\label{2.21a}
\begin{array}{l}
\mbox{$\o^1_K = 1_{W^*_K \times \IR_+}\o$ and $\o^0_K = 1_{(W^*_K)^c \times \IR_+} \o$ are two independent Poisson} 
\\
\mbox{point processes with respective intensity measures $1_{W^*_K \times \IR_+} d \nu\,du$ and}
\\
1_{(W^*_K)^c \times \IR_+} d \nu\,du\,.
\end{array}
\end{equation}

\n
For each $z \in S(0,L)$, the ``surface of $K$'', we now pick a nearest neighbor loop in $K$ starting and 
ending at $z$, and passing through $0$. We then define a map $\varphi$ from $W^*_K$ into itself such 
that for $w^* \in W^*_K$, $\varphi(w^*)$ is the trajectory (modulo time-shift) obtained by ``inserting 
in $w^*$'' just after the entrance in $K$, the loop attached to the entrance point of $w^*$ in $K$. 
The map $\varphi$ is in fact injective and one checks with (\ref{1.25}), (\ref{1.26}) that the image 
measure $\varphi \circ (1_{W^*_K} \nu)$ is absolutely continuous with respect to $1_{W^*_K} \nu$. We 
extend $\varphi$ to $W^*$, by letting $\varphi$ be the identity map on $(W^*_K)^c$. It now follows 
from the above observations that the measurable map $\Phi$ from $\Omega$ into itself defined by:
\begin{equation*}
\Phi(\o) = \dsl_{u_i \le u} \delta_{(\varphi(w_i^*), u_i)} + \dsl_{u_i > u} \delta_{(w_i^*,u_i)}, \;\mbox{for} \;\o = \dsl_{i \ge 1} \delta_{(w_i^*,u_i)}\,,
\end{equation*}
is such that
\begin{equation}\label{2.22a}
\mbox{$\Phi \circ \IP$ is absolutely continuous with respect to $\IP$}\,.
\end{equation}

\medskip\n
By construction $\Phi(\o)$ links together all infinite connected components of $\wt{\cI}^u(\o)$, which intersect $K$, and hence $\Phi(A) \subseteq \{N_u < k\}$, where $A$ appears below (\ref{2.20a}). We thus find that
\begin{equation}\label{2.23a}
\Phi \circ(1_A \IP) [N_u < k] = \IP [A \cap \Phi^{-1}(N_u < k)] = \IP [A] > 0\,,
\end{equation}

\medskip\n
and due to (\ref{2.22a}) we see that $\IP[N_u < k] > 0$, a contradiction. This proves (\ref{2.20a}). 
The claim (\ref{2.16b}) will now follow once we show that
\begin{equation}\label{2.24a}
\IP[N_u = \infty] = 0\,.
\end{equation}

\medskip\n
The heart of the matter, cf.~\cite{Grim99}, p.~199, or \cite{HaggJona06}, p.~297, 
is to show that with positive $\IP$-probability there is a trifurcation in $\wt{\cI}^u$, 
i.e. a site $x \in \cI^u(\o)$ with exactly three $\wt{\cI}^u(\o)$-neighbors and the removal of $x$ splits the infinite connected component of $x$ determined by $\wt{\cI}^u(\o)$ in exactly three infinite components.

\medskip
Assume by contradiction that $\IP[N_u = \infty] = 1$, then for arbitrarily large $L > 0$, one has with $K = B(0,L)$,
\begin{equation}\label{2.25a}
\begin{array}{l}
\mbox{$\IP[K$ intersects more than $4 |B(0,100)|$ infinite connected components}
\\
\mbox{of $\wt{\cI}^u(\o)] > 0$}\,.
\end{array}
\end{equation}

\medskip\n
We fix $L$ large enough, such that (\ref{2.25a}) holds and for any three couples of points $(z_1,z_2)$, $(z_3,z_4)$, $(z_5,z_6)$ on the $|\cdot|_\infty$-sphere $S(0,L)$, for which no point in a given pair may be within $|\cdot |_\infty$-distance 100 from any other pair (but points within a pair may be arbitrarily close or even coincide), we can construct $\tau_1,\tau_2,\tau_3$ finite nearest neighbor trajectories in $K$ with respective starting points $z_1,z_3,z_5$ and end points $z_2,z_4,z_6$, so that any two trajectories only meet in $0$, each trajectory visits $0$ only once, and this occurs by crossing an edge touching $0$ and immediately crossing the same edge in the reverse direction. With (\ref{2.21a}) and (\ref{2.25a}) we see that
\begin{equation*}
\begin{array}{l}
\mbox{$\IP \otimes (1_{W^*_K} \nu)^{\otimes m} \big[K$ intersects more than $4 |B(0,100)|$ infinite connected components}
\\
\mbox{determined by $\wt{\cI}^u(\o^0_K + \sum^m_{i=1} \delta_{(w_i^*,u)})\big] > 0$, for some $m > 4 \,|B(0,100)|$}\,,
\end{array}
\end{equation*}

\n
where $\o$ and $w^*_i$, $1 \le i \le m$, are the respective $\Omega$- and $W^*_K$-valued coordinates on the product space, and we use the notation from (\ref{2.18new}) and (\ref{2.21a}). On the above event we can select  three trajectories within the $w_1^*,\dots,w^*_m$ with supports lying in distinct infinite connected components and corresponding pairs of entrance and last exit points of $K$ with mutual $|\cdot |_\infty$-distance bigger than $100$. As a result we see that
\begin{equation}\label{2.26a}
\mbox{$\IP \otimes (1_{W^*_K} \nu)^{\otimes m} [C_m] > 0$, for some $m \ge 3$}\,,
\end{equation}

\n
where $C_m$ stands for the event
\begin{equation*}
\begin{array}{l}
\mbox{$\big\{\wt{\cI}^u (\o^0_K + \sum^m_{i=1} \,\delta_{(w^*_i,u)})$ has at least three infinite connected components}
\\
\mbox{meeting $K$ respectively containing $w^*_{i_1}(\IZ)$, $w^*_{i_2}(\IZ), w^*_{i_3}(\IZ)$ for some distinct}
\\
\mbox{$i_1,i_2,i_3$ in $\{1, \dots, m\}$, and the three corresponding pairs of entrance and last exit}
\\
\mbox{points of $K$ have mutual $|\cdot|_\infty$-distance bigger than $100\big\}$}.
\end{array}
\end{equation*}

\medskip\n
Observe now that without loss of generality we can assume $m=3$ in (\ref{2.26a}).

\medskip
We denote with $\gamma$ the map from $(W^*_K)^3$ into itself such that 
$\gamma(w_1^*,w_2^*,w_3^*) = (\ov{w}_1^{\,*},\ov{w}_2^{\,*},\ov{w}_3^{\,*})$, 
where $\gamma$ simply coincides with the identity if the three pairs of entrance 
and last exit points for $K$ for $w_1^*, w_2^*, w_3^*$ do not fulfill the condition 
appearing below (\ref{2.25a}), and otherwise such that 
$\ov{w}_1^{\,*},\ov{w}_2^{\,*},\ov{w}_3^{\,*}$ are obtained from $w_1^*,w_2^*,w_3^*$ by 
replacing the respective portions of trajectory between first entrance in $K$ and last 
exit from $K$ by $\tau_1,\tau_2,\tau_3$. With (\ref{1.25}), (\ref{1.26}) one checks that
\begin{equation}\label{2.27a}
\mbox{$\gamma \circ (1_{W^*_K} \nu)^{\otimes 3}$ is absolutely continuous with respect to $(1_{W^*_K} \nu)^{\otimes 3}$}\,. 
\end{equation}

\medskip\n
Note that on the event $C_3$, $0$ is a trifurcation point for $\wt{\cI}^u\big(\o^0_K + \sum^3_{i=1} \delta(\ov{\o}^*_i,u)\big)$, where the notation is the same as in the above paragraph. With a similar calculation as in (\ref{2.23a}) 
we see that
\begin{equation*}
\mbox{$\IP \otimes (1_{W^*_K} \nu)^{\otimes 3} \big[0$ is a trifurcation point for $\wt{\cI}^u (\o^0_K + \sum^3_{i=1} \delta_{(w^*_i,u)})\big] > 0$}\,.
\end{equation*}

\n
With (\ref{2.21a}) this readily implies that
\begin{equation}\label{2.28a}
\mbox{$\IP[0$ is a trifurcation point for $\wt{\cI}^u(\o)] > 0$}\,.
\end{equation}

\medskip\n
The proof of (\ref{2.24a}) now runs just as in \cite{Grim99}, p.~200-202. This concludes the proof of (\ref{2.16b}). 

\end{proof}

\medskip
Just as in the case of Bernoulli percolation, cf.~\cite{Grim99}, p.~13, we can introduce the critical value
\begin{equation}\label{2.17a}
u_* = \inf\{u \ge 0, \,\eta(u) = 0\} \in [0,\infty]\,.
\end{equation}

\medskip\n
Is this critical value non-degenerate? We will see in Section 3 that $u_* < \infty$, cf.~Theorem \ref{theo3.5}, and in Section 4 that $u_* > 0$, as soon as $d \ge 7$, cf.~Theorem \ref{theo4.3}.

\medskip
We are now going to discuss the exponential bound mentioned at the beginning of this section. For $1 \le m \le d$, we write $\cL_m$ for the collection of $m$-dimensional affine subspaces of $\IZ^d$ generated by $m$ distinct vectors of the canonical basis $(e_i)_{1 \le i \le d}$ of $\IR^d$:
\begin{equation}\label{2.18}
\begin{split}
\cL_m = \big\{F \subseteq \IZ^d; &\;\;  \mbox{for some $I \subseteq \{1,\dots,d\}$ with $|I| = m$ and some $y \in \IZ^d$}, 
\\
&\ \;\;F= y + \dsl_{i \in I} \IZ \,e_i\big\}\,,
\end{split}
\end{equation}
and introduce
\begin{equation}\label{2.19}
\mbox{$\cA_m =$ the collection of finite subsets $A$ with $A \subseteq F$ for some $F \in \cL_m$}\,.
\end{equation}

\n
We denote with $q(\nu)$ the return probability to the origin of simple random walk in $\IZ^\nu$, i.e. with hopefully obvious notation:
\begin{equation}\label{2.20}
q(\nu) = P_0^{\IZ^\nu} [\wt{H}_0 < \infty], \;\mbox{for $\nu \ge 1$}\,.
\end{equation}

\n
The promised exponential estimate comes in the following 
\begin{theorem}\label{theo2.4} $(d \ge 4$, $1 \le m \le d-3)$

\medskip
Assume that $\lambda > 0$ satisfies
\begin{equation}\label{2.21}
\chi(\lambda) \stackrel{\rm def}{=} e^\lambda \Big(\dis\frac{m}{d} + \Big(1 - \dis\frac{m}{d}\Big) \,q(d-m)\Big) < 1\,,
\end{equation}

\n
then for $u \ge 0$, $A \in \cA_m$ and $A \subseteq K \subset \subset \IZ^d$, with the notation $f_A(w) = \sum_{n \ge 0} 1_{\{X_n(w) \in A\}}$, for $w \in W_+$, one has
\begin{equation}\label{2.22}
\IE[\exp\{\lambda \langle \mu_{K,u}, f_A \rangle\}] \le \exp\Big\{ u \,{\rm cap} (A) \;\dis\frac{e^\lambda -1}{1 - \chi(\lambda)}\Big\}\,,
\end{equation}
and the left-hand side does not depend on $K$ as above. 

\medskip
Moreover there exists $u_1(d,m,\lambda) > 0$, such that:
\begin{equation}\label{2.23}
\IP[\cI^u \supseteq A] \le \exp\{- \lambda \,|A|\}, \;\mbox{for all $A \in \cA_m$ and $u \le u_1$}\,.
\end{equation}
\end{theorem}

\begin{proof}
Consider $A \in \cA_m$, $F \in \cL_m$ containing $A$, then for $A \subseteq K \subseteq K^\prime \subset \subset \IZ^d$, we find that 
\begin{equation*}
\langle \mu_{K,u}, f_A \rangle \stackrel{(\ref{1.21}) {\rm i)}}{=} \langle \mu_{K^\prime,u}, f_A \circ \theta_{H_K} 1\{H_K < \infty\}\rangle = \langle \mu_{K^\prime,u}, f_A \rangle\,.
\end{equation*}

\medskip\n
So the left-hand side of (\ref{2.22}) does not depend on $K \subset \subset \IZ^d$ containing $A$. In particular picking $K = A$, we find that it equals
\begin{equation}\label{2.24}
\IE[\exp\{\lambda \langle \mu_{A,u}, f_A \rangle \}] \stackrel{(\ref{1.20}),(\ref{1.43})}{=} \exp\{u \,E_{e_A} [e^{\lambda f_A} -1]\}\,.
\end{equation}
Introducing the function
\begin{equation}\label{2.25}
\phi(x) = E_x [e^{\lambda f_A}], \;\mbox{for $x \in \IZ^d$}\,,
\end{equation}

\n
and writing $R_F \stackrel{\rm def}{=} T_F + H_F \circ \theta_{T_F}$, the return time to $F$, see (\ref{1.3}) for notation, we find:
\begin{align*}
e^{\lambda f_A} & \le e^{\lambda T_F} \big(1_{\{R_F = \infty\}} + 1_{\{R_F < \infty\}} \,e^{\lambda f_A} \circ \theta_{R_F}\big)
\\
& = e^{\lambda T_F} \big(1 +1_{\{R_F < \infty\}} (e^{\lambda f_A} \circ \theta_{R_F} - 1)\big)\,.
\end{align*}

\medskip\n
With the strong Markov property at times $R_F$ and then $T_F$, we thus obtain:
\begin{equation}\label{2.26}
\begin{split}
\phi(x) & \le E_x[e^{\lambda T_F}] + E_x\big[e^{\lambda T_F} P_{X_{T_F}}[H_F < \infty]\big] (\|\phi\|_\infty -1 )
\\[1ex]
&\hspace{-1ex} \stackrel{(\ref{2.20})}{=} E_x [e^{\lambda T_F}] \big(1 + q(d-m) (\|\phi\|_\infty - 1)\big)\,,
\end{split}
\end{equation}

\medskip\n
considering in the last step the motion of the walk in the components ``transversal'' to $F$. Note that when $z \notin F$, $T_F = 0$, $P_z$-a.s., whereas when $z \in F$, $T_F$ has geometric distribution with success probability $1 - \frac{m}{d}$. Hence with $\lambda$ satisfying (\ref{2.21}) we find that:\begin{equation}\label{2.27}
E_z [\exp\{\lambda T_F\}] = \dsl_{k \ge 1} \;\Big(1 - \mbox{\f $\dis\frac{m}{d}$}\Big) \Big( \mbox{\f $\dis\frac{m}{d}$}\Big)^{k-1} \,e^{\lambda k} = e^\lambda \Big(1 - \mbox{\f $\dis\frac{m}{d}$}\Big)\Big(1 - e^\lambda \; \mbox{\f $\dis\frac{m}{d}$}\Big)^{-1} \stackrel{\rm def}{=}\alpha \,.
\end{equation}
With a routine approximation argument of $f_A$ by a finite sum, to exclude the possibility that $\|\phi \|_\infty$ is infinite, and (\ref{2.26}) we see that:
\begin{equation*}
\|\phi\|_\infty \le\dis\frac{\alpha(1- q(d-m))}{1-q(d-m)\alpha} \;,
\end{equation*}
and hence
\begin{equation}\label{2.28}
\|\phi\|_\infty  - 1 \le \dis\frac{\alpha-1}{1 - q(d-m)\alpha} = \mbox{\f $\dis\frac{e^\lambda - 1}{1 - \chi(\lambda)}$}  \;.
\end{equation}

\medskip\n
Coming back to (\ref{2.24}), and using (\ref{1.7}) we find (\ref{2.22}). As for (\ref{2.23}), note with (\ref{1.6}), (\ref{1.62}) that
\begin{equation*}
{\rm cap} (A) \le \dsl_{x \in A} {\rm cap}(\{x\}) = \dis\frac{|A|}{g(0)} \;.
\end{equation*}

\n
Further on the event $\{\cI^u \supseteq A\}$ we have $\langle \mu_{A,u}, f_A \rangle \ge |A|$. So choosing $\wt{\lambda}(d,m,\lambda) > \lambda$, such that $1 - \chi(\wt{\lambda}) = \frac{1}{2} \;(1 - \chi(\lambda))$ we now see that for $A \in \cA_m$:
\begin{equation}\label{2.29}
\begin{split}
\IP[\cI^u \supseteq A] & \stackrel{(\ref{2.22})}{\le} \exp\Big\{ - \wt{\lambda} \,|A| + \dis\frac{|A|}{g(0)}\; u \;\dis\frac{e^{\wt{\lambda}} - 1}{ 1- \chi(\wt{\lambda})}\Big\}
\\
& \;\; \le \exp\{ - \lambda \,|A|\}\,,
\end{split}
\end{equation}

\n
if $u \le u_1(d,m,\lambda)$. This proves (\ref{2.23}).
\end{proof}

\begin{remark}\label{rem2.5} ~ \rm

\medskip\n
1) The proof of Theorem \ref{theo2.4} is very similar to the proofs of Theorem 2.1 of \cite{BenjSzni06} and Theorem 1.2 of \cite{DembSzni08}, however it has a somewhat more algebraic character due to the nature of the basic model we work with. 

\bigskip\n
2) There is no bound of type (\ref{2.23}) valid uniformly for $\cA_d$ the collection of subsets of $\IZ^d$. The argument is in essence the same as in Remark 2.4 2) of \cite{BenjSzni06}. One can for instance consider $A_L = B(0,L)$ and note that for large $L$, when the random walk starts in $A_L$, conditionally on not leaving $A_{2L}$ up to time $c\,L^d \log L$ (with $c$ a large enough constant), it covers $A_L$ with probability at least $\frac{1}{2}$, cf.~(2.33) of \cite{BenjSzni06}. From this it follows that for large $L$,
\begin{align*}
\IP[\cI^u \supseteq A_L] & \ge \IP[\mu_{A_{L},u} \not= 0] \;\fr \;\inf\limits_{x \in A_L} \;P_x[T_{A_{2L}} > c \,L^d \log L]
\\
& \ge c(1 - \exp\{ - u\, {\rm cap}(A_L)\}) \,\exp\{-c\,L^{d-2} \log L\}\,.
\end{align*}

\n
As a result no matter how small $u > 0$, one finds that
\begin{equation}\label{2.30}
\lim\limits_{L \r \infty} |A_L|^{-1} \log \IP[\cI^u \supseteq A_L] = 0\,.
\end{equation}

\n
3)  One can combine (\ref{2.23}) with a Peierls-type argument by considering the collection of $*$-nearest neighbor circuits separating $0$ from infinity in some $F \in \cL_2$ containing $0$, cf.~Corollary 2.5 of \cite{BenjSzni06} or Corollary 1.5 of \cite{DembSzni08}, and see that when $d$ satisfies
\begin{equation}\label{2.31}
7 \Big(\mbox{\f $\dis\frac{2}{d}$} + \Big( 1- \mbox{\f $\dis\frac{2}{d}$}\Big) \;q(d-2)\Big) < 1\,,
\end{equation} 
then for small $u > 0$, $\cV^u$ percolates i.e.
\begin{equation}\label{2.32}
\IP[{\rm Perc}(u)] = 1, \;\mbox{for small $u > 0$}\,.
\end{equation}

\medskip\n
The factor $7$ in (\ref{2.31}) simply stems from the fact that there are at most $8\, 7^{n-1}$ $*$-nearest neighbor circuits with $n$ steps in $\IZ^2$ that start at the origin. It is known that $q(\nu) \sim (2 \nu)^{-1}$, as $\nu \r \infty$, cf.~(5.4) of \cite{Mont56}, and hence (\ref{2.31}) holds for large $d$. Clearly (\ref{2.31}) forces $d > 14$, and with the help of tables of values for $q(\cdot)$, one can see that in effect (\ref{2.31}) holds exactly when $d \ge 18$, cf.~Remark 2.1 of \cite{DembSzni08}. In section 4 we will show that (\ref{2.32}) holds when $d \ge 7$. \hfill $\square$
\end{remark}

 \section{Absence of percolation for large $u$}
 \setcounter{equation}{0}
 
The principal object of this section is to show in Theorem \ref{theo3.5} that when $d \ge 3$, for large enough $u$, $\IP$-almost surely all connected components of $\cV^u$ are finite. We know from Remark \ref{rem1.6} 1) or (\ref{1.57}) that in general $\IP[\cV^u \supseteq A]$ does not decay exponentially with $|A|$. This creates an obstruction to the classical Peierls-type argument, which is used in the context of Bernoulli percolation. It substantially complicates the matter. The strategy of the proof we present here is instead based on a renormalization argument. We establish in Proposition \ref{prop3.1} key estimates on the probability of existence of certain crossings at scale $L_n$ in $\cV^{u_n}$, cf.~(\ref{3.7}), (\ref{3.8}), on an increasing sequence of length scales $L_n$ and an increasing but bounded sequence of values $u_n$. The proof of Proposition \ref{prop3.1} uses a recurrence propagating certain controls, cf.~(\ref{3.10}), from one scale to the next along a sequence of level-values as in (\ref{3.9}). Once Proposition \ref{prop3.1} is established it is a simple matter to deduce Theorem \ref{theo3.5}. We will now introduce some notation.
 
 \medskip
 We consider the positive number $a$ and an integer $L_0$:
 \begin{equation}\label{3.1}
 a = \mbox{\f $\dis\frac{1}{100d}$}, \;L_0 > 1\,.
 \end{equation}
 
 \medskip\n
 We then define an increasing sequence of length scales via
 \begin{equation}\label{3.2}
 L_{n+1} = \ell_n\,L_n, \;\mbox{where $\ell_n = 100[L^a_n] (\ge L^a_n)$, for $n \ge 0$}\,,
 \end{equation}
 
 \medskip\n
 so that $L_n, n \ge 0$, quickly grows to infinity:
 \begin{equation}\label{3.3}
 L_n \ge L_0^{(1+a)^n}, \;\mbox{for $n \ge 0$}\,.
 \end{equation}
 
\medskip\n
 We organize $\IZ^d$ in a hierarchical way with $L_0$ corresponding to the bottom scale and $L_1 < L_2 < \dots$ representing coarser and coarser scales. For this purpose, given $n \ge 0$, we consider the set of labels at level $n$:
 \begin{equation}\label{3.4}
I_n = \{n\} \times \IZ^d\,.
 \end{equation}
 
\medskip\n
  To each label at level $n$, $m = (n,i) \in I_n$, we associate the boxes:
 \begin{equation}\label{3.5}
 \begin{split}
 C_m & = \Big(iL_n + [0, L_n)^d \Big) \cap \IZ^d\,,
 \\
 \wt{C}_m & =  \bigcup\limits_{m^\prime \in I_n: d(C_{m^\prime}, C_m) \le 1} C_{m^\prime}\,,
 \end{split}
 \end{equation}
 
 \medskip\n
where we refer to the notation above (\ref{1.1}).  It is straightforward to see that $C_m, m \in I_n$, is a partition of $\IZ^d$ into boxes of side-length $L_n - 1$, and $\wt{C}_m$ simply stands for the union of $C_m$ and its ``$*$-neighboring'' boxes of level $n$. Also when $m \in I_{n+1}$, then $C_m$ is the disjoint union of the $\ell_n^d$ boxes $C_{\ov{m}}$ at level $n$ it contains. We denote with $\wt{S}_m$ the interior boundary of $\wt{C}_m$:
  \begin{equation}\label{3.6}
 \mbox{$\wt{S}_m = \partial_{\rm int} \wt{C}_m$, for $m \in I_n, \,n \ge 0$}\,.
 \end{equation}
 
 \medskip\n
 In what follows we investigate the probability of the existence of certain vacant crossings defined for $u \ge 0, n \ge 0, m \in I_n$, via:
  \begin{equation}\label{3.7}
 \mbox{$A^u_m = \{\o \in \Omega$; there is a nearest neighbor path in $\cV^u(\o) \cap \wt{C}_m
 $ from $C_m$ to $\wt{S}_m\}$}\,.
 \end{equation}
 
\n
 It follows from translation invariance, cf.~Theorem \ref{theo2.1} or (\ref{1.48}) that
 \begin{equation}\label{3.8}
 p_n(u) = \IP[A^u_m], \;u \ge 0, \,n \ge 0, \; \mbox{with} \;m \in I_n\,,
 \end{equation}
 
\medskip\n
 is well-defined, i.e. does not depend on which $m \in I_n$ enters the right-hand side. Clearly the functions $p_n(\cdot)$ are non-increasing on $\IR_+$. Our main task consists in the derivation of recurrence relations on the functions $p_n(\cdot)$. The key control is provided by the following
 
 \begin{proposition}\label{prop3.1} $(d \ge 3$)
 
 \medskip
There exist positive constants $c_1,c_2$, cf.~{\rm (\ref{3.34}), (\ref{3.52})}, such that defining for $u_0 > 0$ and $r \ge 1$ integer
\begin{equation}\label{3.9}
u_n = u_0 \prod\limits_{0 \le n^\prime < n} (1 + c_1 \,\ell_{n^\prime}^{-(d-2)})^{r+1}, \;\mbox{for $n \ge 0$}\,,
\end{equation}

\medskip\n
then for $L_0 \ge c$, $u_0 \ge c(L_0), \;r \ge c(L_0,u_0)$, one has
\begin{equation}\label{3.10}
c_2 \,\ell_n^{2(d-1)} p_n(u_n) \le L^{-1}_n, \;\mbox{for all $n \ge 0$} \,.
\end{equation}
 \end{proposition}

\medskip
 \begin{proof}
In the course of the proof of Proposition \ref{prop3.1}, we will use the expression ``for large $L_0$'', in place of ``for $L_0 \ge c$'', with $c$ a positive constant as explained at the end of the Introduction. We first consider $n \ge 0$, $m \in I_{n+1}$, as well as $0 < u^\prime < u$. We are first going to bound $p_{n+1}(u)$ in terms of $p_n(u^\prime)$, when $\frac{u^\prime}{u}$ is sufficiently away from 1, cf.~(\ref{3.45}), (\ref{3.52}).

\medskip
We write $\cH_1$ for the collection of labels at level $n$ of boxes contained in $C_m$ touching $\partial_{\rm int} C_m$:
 \begin{equation}\label{3.11}
 \cH_1  = \mbox{$\{\ov{m} \in I_n; \;C_{\ov{m}} \subseteq C_m$ and $C_{\ov{m}} \cap \partial_{\rm int} C_m \not= \phi\}$},  
 \end{equation}
as well as
\begin{equation}\label{3.12}
 \cH_2  = \Big\{\ov{m} \in I_n; \; C_{\ov{m}} \cap \Big\{z\in \IZ^d: d(z, C_m) = \mbox{\f $\dis\frac{L_{n+1}}{2}\Big\}$} \not= \phi\Big\}\,, 
 \end{equation}
 
\medskip\n
 for the collection of labels of $n$-level boxes containing some point at $|\cdot|_\infty$-distance $\frac{L_{n+1}}{2}$ from $C_m$ (with a similar notation as above (\ref{1.1})).
 
 \medskip
\begin{center}
\psfragscanon
\includegraphics{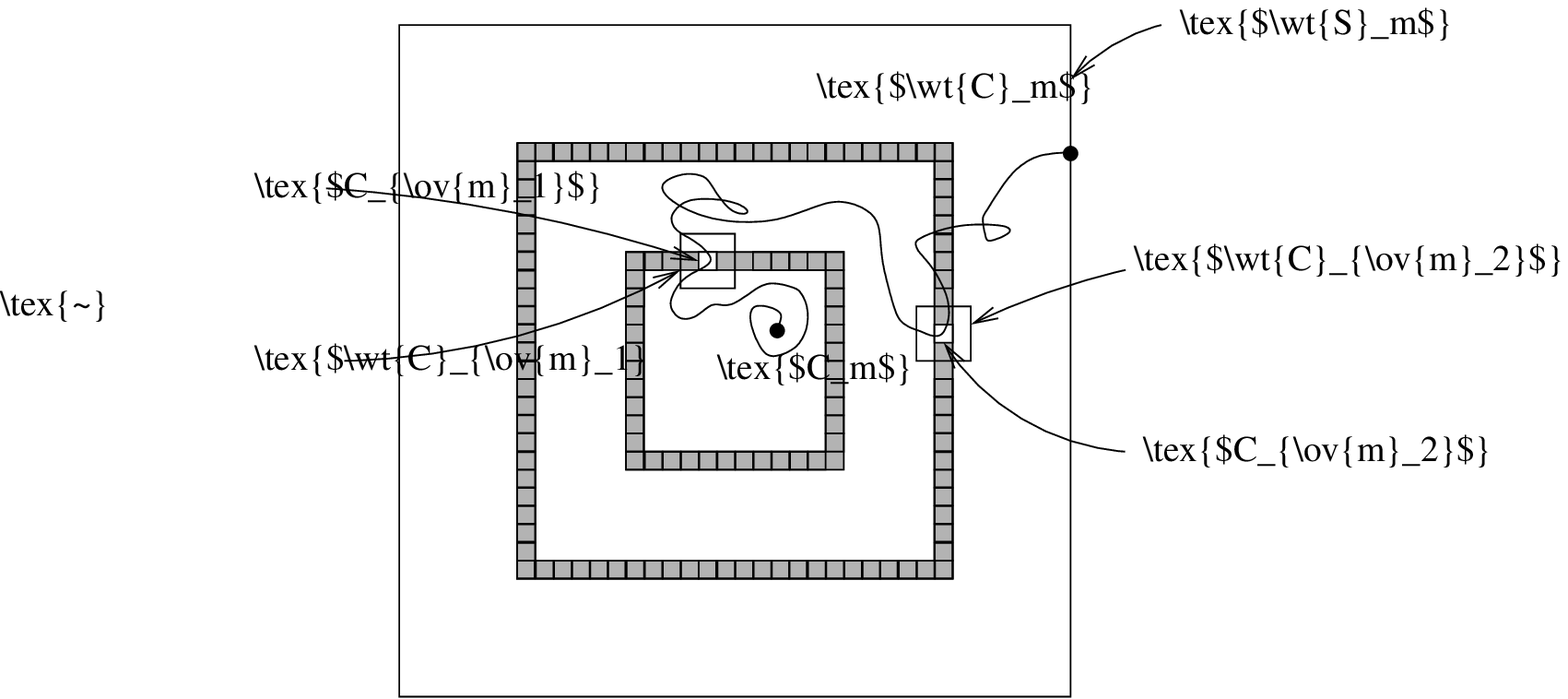}
 \end{center}

\begin{center}
 Fig.~1: ~A schematic illustration of the event $A^u_m$. The path drawn lies in $\cV^u$.
 \end{center}

\n
Observe that any nearest neighbor path in $\cV^u$ originating in $C_m$ and ending in $\wt{S}_m$ must go through some $C_{\ov{m}_1}, \ov{m}_1 \in \cH_1$, reach $\wt{S}_{\ov{m}_1}$, and then go through some $C_{\ov{m}_2}$, $\ov{m}_2 \in \cH_2$, and reach $\wt{S}_{\ov{m}_2}$. Therefore we see that
 \begin{equation}\label{3.13}
 p_{n+1}(u) \le \dsl_{\ov{m}_1 \in \cH_1, \ov{m}_2 \in \cH_2} \IP[A^u_{\ov{m}_1} \cap A^u_{\ov{m}_2}] \le c \,\ell_n^{2(d-1)} \sup\limits_{\ov{m}_1 \in \cH_1, \ov{m}_2 \in \cH_2} \IP[A^u_{\ov{m}_1} \cap A^u_{\ov{m}_2}]\,,
 \end{equation}
 
\n
 using a rough counting argument to bound $|\cH_1|$ and $|\cH_2|$ in the last step. We will now focus our attention on the probability which appears in the last member of (\ref{3.13}). We write $V = \wt{C}_{\ov{m}_1} \cup \wt{C}_{\ov{m}_2}$ for given $\ov{m}_1 \in \cH_1$, $\ov{m}_2 \in \cH_2$, and just as in (\ref{2.8}) introduce the decomposition
 \begin{equation}\label{3.14}
 \mu_{V,u} = \mu_{1,1} + \mu_{1,2} + \mu_{2,1} + \mu_{2,2}\,,
 \end{equation}
 
 \medskip\n
 where $\wt{C}_{\ov{m}_1}$, $\wt{C}_{\ov{m}_2}$ respectively play the role of $K$ and $K+x$ in (\ref{2.8}). In particular $\mu_{i,j}$, $1 \le i, j \le 2$, are independent Poisson point processes on $W_+$ with intensity measures $\gamma_{i,j}$, $1 \le i, j \le 2$, as in (\ref{2.10}). The following notation will be convenient. When $\Lambda$ is a random point process on $W_+$ defined on $\Omega$, i.e. a measurable map from $\Omega$ into the space of pure point measures on $W_+$, we denote with $A_{\ov{m}}(\Lambda)$, for $\ov{m} \in I_n$, the event:
\begin{align}
A_{\ov{m}}(\Lambda) =  \big\{\o \in \Omega; &\; \mbox{there is a nearest neighbor path in $\wt{C}_{\ov{m}} \backslash \big(\bigcup\limits_{w \in {\rm Supp}(\Lambda(\o))} w(\IN)\big)$}\label{3.15}
\\[-2ex]
&\;\mbox{from $C_{\ov{m}}$ to $\wt{S}_{\ov{m}}\big\}$} \,.\nonumber
\end{align}
 
 \medskip\n
 For instance with (\ref{1.54}) we see that for any $\ov{m} \in I_n$:
 \begin{equation}\label{3.16}
 A^u_{\ov{m}} = A_{\ov{m}} (\mu_{K,u}), \;\mbox{for any $K \supseteq \wt{C}_{\ov{m}}$}\;.
 \end{equation}
 
\n
We can apply this identity to $\ov{m} = \ov{m}_i$, $i = 1,2$, with $K = V$. Noting that $w \in {\rm Supp}\,\mu_{2,2}$ implies $w(\IN) \cap \wt{C}_{\ov{m}_1} = \phi$, we find that
 \begin{align*}
 A^u_{\ov{m}_1} \cap A^u_{\ov{m}_2} & = A_{\ov{m}_1} (\mu_{V,u}) \cap A_{\ov{m}_2} (\mu_{V,u})
 \\
 & = A_{\ov{m}_1} (\mu_{1,1} + \mu_{1,2} + \mu_{2,1}) \cap A_{\ov{m}_2} (\mu_{V,u})
 \\
 &\subseteq A_{\ov{m}_1}(\mu_{1,1} + \mu_{1,2} + \mu_{2,1}) \cap A_{\ov{m}_2} (\mu_{2,2})\,.
 \end{align*}
 
\medskip \n
 With the help of the independence properties mentioned above we find that
 \begin{equation}\label{3.17}
 \begin{split}
 \IP\big[A^u_{\ov{m}_1} \cap A^u_{\ov{m}_2}\big] & \le \IP\big[A_{\ov{m}_1} (\mu_{1,1} + \mu_{1,2} + \mu_{2,1})\big] \;\IP\big[A_{\ov{m}_2} (\mu_{2,2})\big]
 \\[0.5ex]
 & = p_n(u) \;\IP\big[A_{\ov{m}_2} (\mu_{2,2})\big]\,.
 \end{split}
 \end{equation}
 
 \n
 Our next task is to bound $\IP[A_{\ov{m}_2}(\mu_{2,2})]$ from above. For this purpose we decompose the $\mu_{i,j}$,  $1 \le i, j \le 2$, in (\ref{3.14}) into
 \begin{equation}\label{3.18}
 \mu_{i,j} = \mu^\prime_{i,j} + \mu^*_{i,j} \;,
 \end{equation}
 
 \n
 where the $\mu^\prime_{i,j}$,  $\mu^*_{i,j}$, $ 1 \le i, j \le 2$, are independent Poisson point processes on $W_+$, with $\mu^\prime_{i,j}$ defined as $\mu_{i,j}$, with $u^\prime ( < u)$ replacing $u$ in (\ref{3.14}), and $\mu^*_{i,j}$ defined analogously as in (\ref{3.14}), but with the role of $\mu_{V,u}(dw)$ replaced by $\mu_V(dw \times (u^\prime,u])$, cf.~(\ref{1.18}), (\ref{1.20}), (\ref{1.45}), which is also a Poisson point process on $W_+$. we write $\gamma^\prime_{i,j}$ and $\gamma^*_{i,j}$, $1 \le i,j \le 2$, for the intensity measures of these point processes, and note that
 \begin{equation}\label{3.19}
 \gamma^*_{2,2}(dw) = (u - u^\prime) \,1\big\{X_0 \in \wt{C}_{\ov{m}_2}, \;H_{\wt{C}_{\ov{m}_1}} = \infty\big\} \,P_{e_V}(dw) \,.
 \end{equation}
 
 \n
 Our aim is to bound from above $\IP[A_{\ov{m}_2}(\mu_{2,2})] = \IP[A_{\ov{m}_2}(\mu^\prime_{2,2} + \mu^*_{2,2})]$ in terms of quantities involving $p_n(u^\prime) = \IP[A_{\ov{m}_2} (\mu^\prime_{2,2} + \mu^\prime_{2,1} + \mu^\prime_{1,2})]$. The rough idea is to try to dominate the influence on $\wt{C}_{\ov{m}_2}$ of $\mu^\prime_{2,1} + \mu^\prime_{1,2}$ by that of $\mu^*_{2,2}$. This is a kind of ``sprinkling technique'' where the discrepancy between $u$ and $u^\prime$ in the form of $\mu^*_{2,2}$ is used to dominate the long range interaction reflected by $\mu^\prime_{2,1} + \mu^\prime_{1,2}$.

\medskip
With this in mind we introduce an integer $r \ge 1$, and further decompose $\mu^\prime_{2,1}$, $\mu^\prime_{1,2}$ and $\mu^*_{2,2}$ into:

\vspace{-5ex}
\begin{equation}\label{3.20}
\begin{split}
\mu^\prime_{2,1} & = \dsl_{1 \le \ell \le r} \rho^\ell_{2,1} + \ov{\rho}_{2,1}, \;\mu^\prime_{1,2} = \dsl_{1 \le \ell \le r} \rho^\ell_{1,2} + \ov{\rho}_{1,2}\,,
\\[1ex]
\mu^*_{2,2} & = \dsl_{1 \le \ell \le r} \rho^\ell_{2,2} + \ov{\rho}_{2,2} \,,
\end{split}
\end{equation} 

\n
where denoting with $R_k, D_k, k \ge 1$, the successive returns to $\wt{C}_{\ov{m}_2}$ and departures from $U = \{z \in \IZ^d; d(z, \wt{C}_{\ov{m}_2}) \le \frac{1}{10} \;L_{n+1}\}$, cf.~(\ref{1.4}) and the notation above (\ref{1.1}), we have set for $1 \le i \not= j \le 2$, $\ell \ge 1$,
\begin{equation}\label{3.21}
\begin{array}{ll}
\rho^\ell_{i,j} \, = 1\{R_\ell < D_\ell  < R_{\ell + 1} = \infty\} \,\mu^\prime_{i,j}, & \ov{\rho}_{i,j} = 1\{R_{r+1} < \infty\} \,\mu^\prime_{i,j}
\\[1ex]
\rho^\ell_{2,2}  = 1\{R_\ell < D_\ell < R_{\ell + 1} = \infty\} \,\mu^*_{2,2}, & \ov{\rho}_{2,2} = 1\{R_{r+1} < \infty\} \,\mu^*_{2,2}\,,
\end{array}
\end{equation} 
(note that $\{R_1 < D_1 < \infty\}$ has full measure under each of $\mu^\prime_{2,1}$, $\mu^\prime_{1,2}$ and $\mu^*_{2,2}$). 

\bigskip
We then see that with the above definitions and the independence property mentioned below (\ref{3.18}),
\begin{equation}\label{3.22}
\begin{array}{l}
\mbox{$\mu^\prime_{2,2}$, $\rho^\ell_{i,j}$, $1 \le \ell \le r$, $\ov{\rho}_{i,j}$, $1 \le i,j \le 2$, with $i$ or $j \not= 1$, are independent}
\\[1ex]
\mbox{Poisson point processes on $W_+$}\,.
\end{array}
\end{equation}

\n
Letting $\ov{\xi}_{2,1}$ and $\ov{\xi}_{1,2}$ stand for the respective intensity measures on $W_+$ of $\ov{\rho}_{2,1}$ and $\ov{\rho}_{1,2}$, we have
\begin{equation}\label{3.23}
\begin{array}{lcl}
\ov{\xi}_{2,1}(W_+) & \hspace{-3ex}= & \hspace{-3ex}u^\prime\,P_{e_V} [X_0 \in \wt{C}_{\ov{m}_2}, \; H_{\wt{C}_{\ov{m}_1}} < \infty, \;R_{r+1} < \infty]
\\
& \hspace{-3ex}\stackrel{(\ref{1.6}), (\ref{1.7})}{\le} & \hspace{-3ex} u^\prime \,{\rm cap}  (\wt{C}_{\ov{m}_2}) \sup\limits_{x \in \wt{C}_{\ov{m}_2}} P_x [R_{r+1} < \infty]\,
\\
& \hspace{-3ex} \le &  \hspace{-3ex}u^\prime\,{\rm cap}(\wt{C}_{\ov{m}_2}) (\sup\limits_{x \in U^c} \,P_x[H_{\wt{C}_{\ov{m}_2}} < \infty])^r\,,
\end{array}
\end{equation}

\n
where we used the strong Markov property at times $D_r, D_{r-1}, \dots,D_1$, in the last step. With the right-hand inequality of (\ref{1.9}) as well as \cite{Lawl91}, p.~31, we thus find that:
\begin{equation*}
\sup\limits_{x \in U^c} \,P_x[H_{\wt{C}_{\ov{m}_2}} < \infty] \le c\,L^{-(d-2)}_{n+1} \;\dis\frac{L^d_n}{L^2_n} \stackrel{(\ref{3.2})}{=} c\,\ell_n^{-(d-2)}\,,
\end{equation*}
and hence
\begin{equation}\label{3.24}
\ov{\xi}_{2,1}(W_+)  \le  u^\prime \,{\rm cap}(\wt{C}_{\ov{m}_2}) (c\,\ell_n^{-(d-2)})^r  \le u^\prime \,c^r \,L_n^{(d-2) - a(d-2)r}\,,
\end{equation}

\bigskip\n
where in the last step we used (\ref{3.2}) as well as the right-hand inequality of the standard capacity estimate:
\begin{equation*}
c\,L^{(d-2)} \le {\rm cap}(B(0,L)) \le c^\prime L^{(d-2)}, \;\mbox{for $L \ge 1$} \,,
\end{equation*}

\medskip\n
(which for instance follows from (\ref{1.8}), (\ref{1.9}) with $K = B(0,L)$, letting $x$ tend to infinity in (\ref{1.8}) and using (\ref{1.9}) to bound $c\,|x|^{-(d-2)}$ cap$(B(0,L)) \sim P_x[H_{B(0,L)} < \infty]$, for \linebreak  $|x| \r \infty$). In a similar way we find that
\begin{equation}\label{3.25}
\begin{split}
\ov{\xi}_{1,2}(W_+) & =  u^\prime\,P_{e_V} \,[X_0 \in \wt{C}_{\ov{m}_1}, \;H_{\wt{C}_{\ov{m}_2}} < \infty, \,R_{r+1} <  \infty]
\\[1ex]
& \le u^\prime \,c^r \,L_n^{(d-2) - a(d-2)r}\,.
\end{split}
\end{equation}

\n
We will now seek to show that the trace left on $\wt{C}_{\ov{m}_2}$ by paths in the supports of $\mu^\prime_{2,1} - \ov{\rho}_{2,1} = \sum_{1 \le \ell \le r} \,\rho^\ell_{2,1}$ and $\mu^\prime_{1,2} - \ov{\rho}_{1,2} = \sum_{1 \le \ell \le r} \,\rho^\ell_{1,2}$ is dominated by the corresponding trace of paths in the support of $\mu^*_{2,2}$. The point processes $\ov{\rho}_{2,1}$ and $\ov{\rho}_{1,2}$ are then viewed as correction terms to be controlled with the help of (\ref{3.24}), (\ref{3.25}).

\medskip
With this perspective we consider the space $W_f$ of finite nearest neighbor paths on $\IZ^d$, and for $\ell \ge 1$, the measurable map $\phi^\ell$ from $\{D_\ell < R_{\ell + 1} = \infty\} \subseteq W_+$ into the product space $W^{\times \ell}_f$ defined by:
\begin{equation}\label{3.26}
\phi^\ell(w) = (w(R_k + \cdot)_{0 \le \point \le D_k - R_k})_{1 \le k \le \ell} \in W_f^{\times \ell} \;\mbox{for} \; w \in \{D_\ell < R_{\ell + 1} = \infty\}\,.
\end{equation}

\bigskip\n
In other words $\phi^\ell(w)$ for $w$ in the above event keeps track of the $\ell$ portions of the trajectory $w$ corresponding to times going from the successive returns to $\wt{C}_{\ov{m}_2}$ up to departure from $U$. We can view the various $\rho^\ell_{i,j}$, $i$ or $j \not= 1$, with $\ell \ge 1$ fixed as point processes on $\{D_\ell < R_{\ell + 1} = \infty\} (\subseteq W_+)$. We then denote with $\wt{\rho}^{\,\ell}_{i,j}$ their respective images under $\phi^\ell$, which are Poisson point processes on $W^{\times \ell}_f$. We write $\wt{\xi}^\ell_{i,j}$ for their corresponding intensity measures. As a result of (\ref{3.22}), we see that
\begin{equation}\label{3.27}
\begin{array}{l}
\mbox{$\mu^\prime_{2,2}, \wt{\rho}^{\ell}_{i,j}, 1 \le \ell \le r, \ov{\rho}_{i,j}, 1 \le i, j \le 2, i$ or $j \not= 1$ are independent}
\\[1ex]
\mbox{Poisson point processes}\,.
\end{array}
\end{equation}

\medskip\n
We will see that when $u^\prime < u$ are sufficiently far apart, cf.~(\ref{3.34}), $\wt{\rho}^{\,\ell}_{2,2}$ has an intensity measure on $W^{\times \ell}_f$ which is bigger than the intensity measure of $\wt{\rho}^{\,\ell}_{2,1} + \wt{\rho}^{\,\ell}_{1,2}$, for $1 \le \ell \le r$. The following lemma will be helpful, we refer to (\ref{3.11}), (\ref{3.12}) and below (\ref{3.20}) for the notation.

\begin{lemma}\label{lem3.2}
For large $L_0$, for all $n \ge 0, m \in I_{n+1}, \ov{m}_1 \in \cH_1, \ov{m}_2 \in \cH_2, x \in \partial U, y \in \partial_{\rm int} \, \wt{C}_{\ov{m}_2}$, we have:
\begin{align}
&P_x[H_{\wt{C}_{\ov{m}_1}} < R_1 < \infty, X_{R_1} = y ] \le c\,\ell^{-(d-2)}_n \,P_x[H_{\wt{C}_{\ov{m}_1}} > R_1, X_{R_1} = y]\label{3.28}
\\[2ex]
&P_x[H_{\wt{C}_{\ov{m}_1}} < \infty, R_1 = \infty] \le c\,\ell_n^{-(d-2)} \,P_x[R_1 = \infty = H_{\wt{C}_{\ov{m}_1}}]\,.\label{3.29}
\end{align}
\end{lemma}

\begin{proof}
We begin with the proof of (\ref{3.28}). We recall the notation introduced below (\ref{1.4}). For $z \in \partial U$, $y \in \partial_{\rm int} \,\wt{C}_{\ov{m}_2}$ we have
\begin{equation}\label{3.30}
\begin{array}{l}
P_z[H_{\wt{C}_{\ov{m}_1}} < R_1 < \infty, X_{R_1} = y] \stackrel{\mbox{\scriptsize strong Markov}}{=} 
\\[1.5ex]
E_z \big[H_{\wt{C}_{\ov{m}_1}} < R_1, P_{X_{H_{\wt{C}_{\ov{m}_1}}}}[R_1 < \infty, X_{R_1} = y]\big] =
\\[1.5ex]
E_z\big[H_{\wt{C}_{\ov{m}_1}} < R_1,  E_{X_{H_{\wt{C}_{\ov{m}_1}}}} [H_{\partial U} < \infty, P_{X_{H_{\partial U}}} [R_1 < \infty, X_{R_1} = y]\big]\big] \,,
\end{array}
\end{equation}

\n
where in the last step we used for $z^\prime \in \wt{C}_{\ov{m}_1}$ the $P_{z^\prime}\,$-\,almost sure identity  $R_1 =$ \linebreak $H_{\partial U} + R_1 \circ \theta_{H_{\partial U}}$, and the strong Markov property at time $H_{\partial U}$. As a result we see that:
\begin{equation}\label{3.31}
\begin{array}{l}
\sup\limits_{z \in \partial U} \,P_z[H_{\wt{C}_{\ov{m}_1}} < R_1 < \infty, X_{R_1} = y] \le \sup\limits_{z \in \partial U} \,P_z [H_{\wt{C}_{\ov{m}_1}} < \infty] \;\times 
\\[2ex]
\sup\limits_{z \in \partial U} \,P_z [R_1 < \infty, X_{R_1} = y]  \stackrel{(\ref{1.9})}{\le} c\,\ell_n^{-(d-2)} \,\sup\limits_{z \in \partial U} \,P_z [R_1 < \infty, X_{R_1} = y] \,,
\end{array}
\end{equation}

\medskip\n
with a similar bound as above (\ref{3.24}) in the last step. Note that
\begin{equation*}
P_z [R_1 < \infty, X_{R_1} = y] = P_z [H_{\wt{C}_{\ov{m}_2}} < \infty, X_{H_{\wt{C}_{\ov{m}_2}}} = y], \;z \in \wt{C}_{\ov{m}_2}^c\,,
\end{equation*}

\n
is a positive harmonic function, and using Harnack's inequality, cf.~Theorem 1.7.2 of \cite{Lawl91}, together with a standard covering argument, we see that:
\begin{equation}\label{3.32}
\sup\limits_{z \in \partial U} \,P_z[R_1 < \infty, X_{R_1} = y] \le c \inf\limits_{z \in \partial U} P_z \,[R_1 < \infty, X_{R_1} = y]\,.
\end{equation}

\medskip\n
Therefore coming back to (\ref{3.31}) we see that
\begin{equation*}
\begin{array}{l}
\sup\limits_{z \in \partial U} \,P_z[H_{\wt{C}_{\ov{m}_1}} < R_1  < \infty, X_{R_1} = y] \le c^\prime \,\ell_n^{-(d-2)}  \inf\limits_{z \in \partial U} \,P_z [R_1 < \infty, X_{R_1} = y] =
\\[1ex]
c^\prime \,\ell_n^{-(d-2)} \inf\limits_{z \in \partial U} (P_z[H_{\wt{C}_{\ov{m}_1}} < R_1 < \infty, X_{R_1} = y]  +  P_z[R_1 < \infty, X_{R_1} = y, H_{\wt{C}_{\ov{m}_1}} > R_1])\,.
\end{array}
\end{equation*}

\n
For large $L_0$, we have $c^\prime \,\ell_n^{-(d-2)}  \le \frac{1}{2}$, for all $n \ge 0$, with $c^\prime$ as in the last line of (\ref{3.32}), and we hence see that for $x \in \partial U$,
\begin{equation}\label{3.33}
P_x[H_{\wt{C}_{\ov{m}_1}} < R_1  < \infty, X_{R_1} = y] \le 2c^\prime  \,\ell_n^{-(d-2)}  P_x[H_{\wt{C}_{\ov{m}_1}} > R_1, X_{R_1} = y] .
\end{equation}

\medskip\n
This proves (\ref{3.28}). We now turn to the proof of (\ref{3.29}) which is more elementary. Indeed one has
\begin{equation*}
\inf\limits_{x \in \partial U} \,P_x[R_1 = \infty, H_{\wt{C}_{\ov{m}_1}} = \infty] \ge c\,,
\end{equation*}

\medskip\n
as follows from the invariance principle used to let the walk move at a distance from $V = \wt{C}_{\ov{m}_1} \cup \wt{C}_{\ov{m}_2}$, which is a multiple of $L_{n+1}$, as well as (\ref{1.9}) and standard bounds on the Green function. On the other hand the left-hand side of (\ref{3.29}) with a similar inequality as above (\ref{3.24}) is bounded by $c \,\ell_n^{-(d-2)}$. Our claim follows.
\end{proof}

The main control on the intensity measure $\wt{\xi}^\ell_{1,2} + \wt{\xi}^\ell_{2,1}$ of $\wt{\rho}^{\,\ell}_{1,2} + \wt{\rho}^{\,\ell}_{2,1}$ in terms of the intensity measure $\wt{\xi}^\ell_{2,2}$ of $\wt{\rho}^{\,\ell}_{2.2}$ is provided by the next
\begin{lemma}\label{lem3.3}
For large $L_0$, one has
\begin{equation}\label{3.34}
\wt{\xi}^\ell_{1,2} + \wt{\xi}^\ell_{2,1} \le \dis\frac{u^\prime}{u - u^\prime} \Big[\Big(1 + \dis\frac{c_1}{\ell_n^{d-2}}\Big)^{\ell+1} - 1\Big]\;\wt{\xi}^\ell_{2,2}, \;\mbox{for $\ell \ge 1$} \,.
\end{equation}
\end{lemma}

\begin{proof}
The measure $\wt{\xi}^\ell_{2,1}$ on $W^{\times \ell}_f$ is the image under $\phi^\ell$ of the intensity measure $\xi^\ell_{2,1}$ on $\{D_\ell < R_{\ell + 1} = \infty\}$ ($\subseteq  W_+$) of $\rho^\ell_{2,1}$, which in view of (\ref{3.21}) equals
\begin{equation}\label{3.35}
\xi^\ell_{2,1}(dw) = u^\prime \,P_{e_V}[dw, X_0 \in \wt{C}_{\ov{m}_2},\,H_{\wt{C}_{\ov{m}_1}} < \infty, \,D_\ell < R_{\ell + 1} = \infty]\,.
\end{equation}

\medskip\n
As a result we find that with hopefully obvious notations
\begin{equation}\label{3.36}
\begin{array}{l}
\wt{\xi}_{2,1}^\ell (dw_1,\dots,dw_\ell)   = 
\\[1ex]
u^\prime P_{e_V} \big[  X_0 \in  \wt{C}_{\ov{m}_2}, H_{ \wt{C}_{\ov{m}_1}} < \infty,  D_\ell < R_{\ell +1} = \infty,  (X_{R_k +\point})_{0 \le \point \le D_k - R_k} \in dw_k,
\\[1ex]
\qquad  \quad 1 \le k \le \ell\big] =
 \\
u^\prime P_{e_V} \big[ X_0 \in \wt{C}_{\ov{m}_2}, \bigcup\limits^\ell_{k=1} \{H_{\wt{C}_{\ov{m}_1}} \circ \theta_{D_k} + D_k < R_{k+1}\}, \,D_\ell < R_{\ell + 1} = \infty\,,
\\
\qquad \quad (X_{R_k + \point})_{0 \le\point \le D_k - R_k} \in dw_k, \,1 \le k \le \ell\big] =
\\[1ex]
u^\prime \dsl_{\phi \not= B \subseteq \{1,\dots,\ell\}} P_{e_V}\big [X_0 \in \wt{C}_{\ov{m}_2}, \,H_{\wt{C}_{\ov{m}_1}} \circ \theta_{D_k} + D_k < R_{k+1}\;\mbox{exactly when} 
\\[-1ex]
\hspace{3.2cm}\mbox{$k \in B$, for} \; 1 \le k \le \ell, D_\ell < R_{\ell + 1}= \infty, 
\\[1ex]
\hspace{3.2cm} (X_{R_k + \point})_{0 \le \point \le D_k - R_k} \in dw_k, 1 \le k \le \ell\big].
 \end{array}
 \end{equation}
  \medskip\n
The generic term of the above sum evaluated on $(w_1,\dots,w_\ell) \in W^{\times \ell}_f$ equals:
\begin{equation}\label{3.37}
\begin{array}{l}
u^\prime P_{e_V} \big[  X_0 \in  \wt{C}_{\ov{m}_2}, H_{ \wt{C}_{\ov{m}_1}} \circ \theta_{D_k} + D_k < R_{k+1}, \;\mbox{exactly when $k \in B$, $1 \le k \le \ell$}, 
\\
\qquad \quad D_\ell < R_{\ell + 1} = \infty,  (X_{R_k + \point})_{0 \le \point \le D_k - R_k} = w_k (\cdot), \,1 \le k \le \ell\big] =
\\[1ex]
u^\prime E_{e_V} \big[  X_0 \in  \wt{C}_{\ov{m}_2}, H_{ \wt{C}_{\ov{m}_1}} \circ \theta_{D_k} + D_k < R_{k+1}, \;\mbox{exactly when}
\\[1ex]
\qquad \quad \mbox{$k \in B \cap\{1,\dots, \ell - 1\}$ for $1 \le k \le \ell - 1, D_\ell < \infty$}, 
\\[1ex]
\qquad \quad (X_{R_{k + \point}})_{0 \le \point \le D_k - R_k} = w_k (\cdot), 1 \le k \le \ell,
\\[1ex]
\qquad \quad E_{X_{D_\ell}} [R_1 = \infty, 1\{\ell \notin B\}\,1\{H_{\wt{C}_{\ov{m}_1}} = \infty\} + 1 \{\ell \in B\}\,1\{H_{\wt{C}_{\ov{m}_1}} < \infty\}]\big],
 \end{array}
 \end{equation}
using the strong Markov property at time $D_\ell$ in the last step.

\medskip
If we denote with $w^s_k$ and $w^e_k$ the starting point and the end point of $w_k$, for $1 \le k \le \ell$, the above expression vanishes unless $w_k^s \in \wt{C}_{\ov{m}_2}$ and $w^e_k \in \partial U$ for each $k \in \{1,\dots,\ell\}$. If these conditions are fulfilled, using the strong Markov property repeatedly at times $R_\ell, D_{\ell - 1}, R_{\ell - 1} \dots D_1$, we see that the last line of (\ref{3.37}) equals:
\begin{equation*}
\begin{array}{l}
u^\prime P_{e_V} \big[  X_0 \in  \wt{C}_{\ov{m}_2}, (X_\point)_{0 \le \point \le D_1} = w_1(\cdot)] \,E_{w^e_1}\big[1 \{ 1\notin B\}\,1\{H_{\wt{C}_{\ov{m}_1}} > R_1\}\; + 
\\[1ex]
1\{1 \in B\}\,1\{H_{\wt{C}_{\ov{m}_1}} < R_1\}, \,R_1 < \infty, X_{R_1} = w_2^s\big] P_{w^s_2}[(X_\point)_{0 \le\point \le D_1} = w_2(\cdot)] \dots
\\[1ex]
E_{w_\ell^e} \big[1 \{\ell \notin B\} \,1\{H_{\wt{C}_{\ov{m}_1}} = \infty\} + 1 \{\ell \in B\} \,1\{H_{\wt{C}_{\ov{m}_1}} < \infty\}, \,R_1 = \infty\big]\,.
\end{array}
\end{equation*}

\medskip\n
We can use Lemma \ref{lem3.2} for all terms in the above expression where $k \in B$, and repeatedly apply the Markov property to come back to an expression similar to (\ref{3.37}). In this fashion we see that the above expression is at most:
\begin{equation}\label{3.38}
\begin{split}
(c\,\ell_n^{-(d-2)})^{|B|}\, u^\prime P_{e_V}\big[ &X_0 \in \wt{C}_{\ov{m}_1}, H_{\wt{C}_{\ov{m}_1}} \circ \theta_{D_k} + D_k \ge R_{k+1}, \;\mbox{for}\;1 \le k \le \ell,
\\[-0.5ex]
& D_\ell < R_{\ell + 1} = \infty, \; (X_{R_k + \point})_{0 \le \point \le D_k - R_k} = w_k(\cdot), \,1 \le k \le \ell\big]\,.
\end{split}
\end{equation}

\medskip\n
Summing over the various non-empty subsets $B$ of $\{1,\dots,\ell\}$, we see with (\ref{3.36}) that
\begin{equation}\label{3.39}
\begin{split}
\wt{\xi}^\ell_{2,1}(dw_1,\dots,dw_\ell) & \le u^\prime \dsl_{\phi \not= B \subseteq \{1,\dots,\ell\}} (c\,\ell_n^{-(d-2)})^{|B|}\,P_{e_V} [X_0 \in \wt{C}_{\ov{m}_2}, H_{ \wt{C}_{\ov{m}_1}} = \infty, 
\\
& \qquad \qquad D_\ell < R_{\ell +1} = \infty, \, (X_{R_{k + \point}})_{0 \le \point \le D_k - R_k} \in d w_k, \,1 \le k \le \ell\big]
\\[1ex]
& = \dis\frac{u^\prime}{u - u^\prime} \;[(1 + c\,\ell_n^{-(d-2)})^{\ell} - 1] \,\wt{\xi}^{\,\ell}_{2,2} (d w_1,\dots,dw_\ell)\,,
\end{split}
\end{equation}

\medskip\n
where we recall that $\wt{\xi}^\ell_{2,2}$ stands for the intensity measure of $\wt{\rho}^{\,\ell}_{2,2}$. 

\medskip
We can proceed in a similar fashion to bound $\wt{\xi}^\ell_{1,2} (dw_1,\dots,dw_\ell)$. The only difference stems from the fact that under $\xi^\ell_{1,2}(dw)$ paths start in $\wt{C}_{\ov{m}_1}$ and $B$, cf.~last line of (\ref{3.36}), can also be the empty set. In an analogous fashion to (\ref{3.38}) we then obtain the following bound:
\begin{equation}\label{3.40}
\begin{array}{l}
u^\prime P_{e_V} \big[ X_0 \in  \wt{C}_{\ov{m}_1}, H_{ \wt{C}_{\ov{m}_1}} \circ \theta_{D_k} + D_k < R_{k+1}, \;\mbox{exactly when $k \in B$, for $1 \le k \le \ell$}, 
\\
\qquad \quad D_\ell < R_{\ell + 1} = \infty,\, (X_{R_k + \point})_{0 \le \point  \le D_k -  R_k} \in w_k (\cdot), \,1 \le k \le \ell\big] \le
\\[1ex]
(c\,\ell_n^{-(d-2)})^{|B|} \,u^\prime P_\rho  \big[  X_0 \in  \wt{C}_{\ov{m}_2}, H_{ \wt{C}_{\ov{m}_1}} \circ \theta_{D_k} + D_k \ge R_{k+1}, \;\mbox{for $1 \le k \le \ell$},
\\[1ex]
 \qquad \qquad \qquad \qquad D_\ell < R_{\ell + 1} = \infty,  \,(X_{R_k + \point})_{0 \le \point \le D_k - R_k} = w_k(\cdot), \,1 \le k \le \ell\big],
 \end{array}
 \end{equation}
with $\rho$ the measure
\begin{equation}\label{3.41}
\begin{split}
\rho(y) & = \dsl_{x \in \wt{C}_{\ov{m}_1}} e_V(x) \,P_x[ H_{\wt{C}_{\ov{m}_2}} < \infty, X_{\wt{C}_{\ov{m}_2}} = y], \;\mbox{for} \;y \in \partial_{\rm int} \,\wt{C}_{\ov{m}_2}\,,
\\[1ex]
& = 0, \;\mbox{otherwise}\,.
\end{split}
\end{equation}

\n
We now see that for $y \in \partial_{\rm int}\,\wt{C}_{\ov{m}_2}$ with a calculation of similar flavor as in (\ref{1.36})
\begin{equation}\label{3.42}
\begin{array}{lcl}
\rho(y) &\hspace{-4ex}  \stackrel{(\ref{1.6})}{=} &\hspace{-4ex}  \dsl_{x \in \wt{C}_{\ov{m}_1}, n \ge 0} P_x[\wt{H}_V = \infty] \,P_x [\wt{H}_{\wt{C}_{\ov{m}_2}} = n, X_n = y]
\\
\\[-2ex]
&\hspace{-4ex}  \stackrel{\rm reversibility}{=} &\hspace{-4ex}  \dsl_{x \in \wt{C}_{\ov{m}_1}, n \ge 0} P_x[\wt{H}_V = \infty] \,P_y [\wt{H}_{\wt{C}_{\ov{m}_2}} > n, X_n = x]
\\
\\[-2ex]
&\hspace{-4ex}  \stackrel{\rm Markov}{=} &\hspace{-4ex}  \dsl_{x \in \wt{C}_{\ov{m}_1}, n \ge 0}  \,P_y [\wt{H}_{\wt{C}_{\ov{m}_2}} > n, X_n = x, \wt{H}_V \circ \theta_n = \infty]
\\
\\[-2ex]
&\hspace{-4ex}   = &\hspace{-4ex}   P_y [\wt{H}_{\wt{C}_{\ov{m}_2}} = \infty, \,H_{\wt{C}_{\ov{m}_1}} < \infty] \,,
\end{array}
\end{equation}

\medskip\n
summing over the time $n$ and location $x$ of the last visit of the path to $\wt{C}_{\ov{m}_1}$ in the last step. Using the strong Markov property at time $T_U$, (recall that $\wt{C}_{\ov{m}_1} \subseteq U^c$), we thus find
\begin{equation}\label{3.43}
\begin{array}{lcl}
\rho(y) &\hspace{-2ex} = &\hspace{-2ex} E_y \big[ \wt{H}_{\wt{C}_{\ov{m}_2}} > T_U, \,P_{X_{T_U}} [H_{ \wt{C}_{\ov{m}_1}} < \infty, \,H_{ \wt{C}_{\ov{m}_2}} = \infty]\big] 
\\
\\[-2ex]
&\hspace{-2ex} \stackrel{(\ref{3.29})}{\le} & \hspace{-2ex} c\,\ell_n^{-(d-2)} E_y \big[\wt{H}_{ \wt{C}_{\ov{m}_2}} > T_U, P_{X_{T_U}} [ H_{ \wt{C}_{\ov{m}_1}} = \infty = H_{ \wt{C}_{\ov{m}_2}}]\big]
\\
\\[-1ex]
&\hspace{-2ex}= &\hspace{-2ex} c\,\ell_n^{-(d-2)} E_y \big[\wt{H}_V > T_U, P_{X_{T_U}} [H_V = \infty]\big] =  c\,\ell_n^{-(d-2)}\,e_V(y)\,,
\end{array}
\end{equation}

\n
for $y \in \partial_{\rm int} \,\wt{C}_{\ov{m}_2} $, using strong Markov property and (\ref{1.6}) in the last step. Therefore summing (\ref{3.40}) over $B \subseteq \{1,\dots,\ell\}$, we find that
\begin{equation}\label{3.44}
\wt{\xi}^\ell_{1,2}(dw_1,\dots,dw_\ell) \le \dis\frac{u^\prime}{u - u^\prime} \;\dis\frac{c}{\ell_n^{d-2}} \;\Big(1 + \dis\frac{c}{\ell_n^{d-2}}\Big)^\ell \;\wt{\xi}^\ell_{2,2} (dw_1,\dots,dw_\ell)\,.
\end{equation}

\medskip\n
Summing (\ref{3.39}) and (\ref{3.44}) we obtain the claim (\ref{3.34}).
\end{proof}

We now suppose $L_0$ large enough so that Lemma \ref{lem3.2} holds, and also that 
\begin{equation}\label{3.45}
u = \Big(1 + \dis\frac{c_1}{\ell_n^{d-2}}\Big)^{r+1}\,u^\prime, \;\Big(\mbox{hence} \;\dis\frac{u^\prime}{u - u^\prime} \,\Big[\Big( 1 + \dis\frac{c_1}{\ell_n^{d-2}}\Big)^{r+1} - 1\big] = 1\Big)\,.
\end{equation}

\medskip\n
We will now derive the promised upper bound on $\IP[A_{\ov{m}_2}(\mu_{2,2})]$ in terms of $p_n(u^\prime) = \IP[A^{u^\prime}_{\ov{m}_2}]$. Observe that the restriction of the interlacement at level $u^\prime$ to $\wt{C}_{\ov{m}_2} $ satisfies:
\begin{equation}\label{3.46}
\begin{split}
\cI^{u^\prime} & \cap \wt{C}_{\ov{m}_2} \stackrel{(\ref{1.54})}{=} \bigcup\limits_{w \in {\rm Supp(\mu_{V,u^\prime})}} w(\IN) \cap \wt{C}_{\ov{m}_2} = \bigcup\limits_{w \in {\rm Supp}(\mu^\prime_{2,2} + \mu^\prime_{2,1} + \mu^\prime_{1,2})} w(\IN) \cap \wt{C}_{\ov{m}_2}
\\
& \qquad \quad\;\, = \;\cI^\prime \cup \wt{\cI} \cup \ov{\cI},  
\end{split}
\end{equation}
where we have set
\begin{equation}\label{3.47}
\begin{split}
\cI^\prime & = \bigcup\limits_{w \in {\rm Supp}(\mu^\prime_{2,2})} w(\IN) \cap \wt{C}_{\ov{m}_2}\,,
\\
\wt{\cI} & = \bigcup\limits_{ 1 \le \ell \le r} \;\bigcup\limits_{(w_1,\dots,w_\ell) \in {\rm Supp}\,\wt{\rho}^{\,\ell}_{1,2} + \wt{\rho}^{\,\ell}_{2,1}} \;\mbox{(range $w_1 \cup \dots \cup$ range $w_\ell) \cap \wt{C}_{\ov{m}_2}$}\,,
\\
\ov{\cI} & = \bigcup\limits_{w \in {\rm Supp} \,\ov{\rho}_{1,2} + \ov{\rho}_{2,1}} w (\IN) \cap \wt{C}_{\ov{m}_2}\,,
\end{split}
\end{equation}

\medskip\n
and we used (\ref{3.20}) together with the fact that for any $\ell \ge 1$, $w \in {\rm Supp} \,\rho^\ell_{1,2} \cup {\rm Supp} \,\rho^\ell_{2,1}$, with $\phi^\ell(w) = (w_1,\dots,w_\ell)$ due to (\ref{3.26}): 
\begin{equation*}
w(\IN) \cap \wt{C}_{\ov{m}_2} = \mbox{(range $w_1 \cup \dots \cup$ range $w_\ell) \cap \wt{C}_{\ov{m}_2}$}\,.
\end{equation*}

\n
If we now define $\cI^*$ by replacing $\wt{\rho}_{1,2}^{\,\ell} + \wt{\rho}^{\,\ell}_{2,1}$ in the second line of (\ref{3.47}) by $\wt{\rho}^{\,\ell}_{2,2}$, we see from (\ref{3.27}) that
\begin{equation}\label{3.48}
\mbox{the random sets $\cI^\prime$, $\wt{\cI}$, $\ov{\cI}$, $\cI^*$ are independent under $\IP$}\,.
\end{equation}

\medskip\n
We also see from (\ref{3.34}), (\ref{3.27}) and the choice (\ref{3.45}) that for each $1 \le \ell \le r$, the Poisson point process $\wt{\rho}^{\,\ell}_{1,2} + \wt{\rho}^{\,\ell}_{2,1}$ is stochastically dominated by $\wt{\rho}^{\,\ell}_{2,2}$ so that
\begin{equation}\label{3.49}
\mbox{$\wt{\cI}$ is stochastically dominated by $\cI^*$}\,.
\end{equation}
With (\ref{3.48}), (\ref{3.49}) we thus find in view of (\ref{3.15}) that
\begin{equation}\label{3.50}
\begin{array}{l}
\IP\Big[A_{\ov{m}_2}\Big(\mu^\prime_{2,2} + \dsl_{1 \le \ell \le r} \rho^\ell_{2,2}\Big)\Big]  = 
\\[1ex]
\IP\big[\mbox{there is a crossing in $\wt{C}_{\ov{m}_2} \backslash (\cI^\prime \cup \cI^*)$ from $C_{\ov{m}_2}$ to $\wt{S}_{\ov{m}_2}\big]$}
\\[1ex]
\le \IP\big[\mbox{there is a crossing in $\wt{C}_{\ov{m}_2} \backslash (\cI^\prime \cup \wt{\cI})$ from $C_{\ov{m}_2}$ to $\wt{S}_{\ov{m}_2}\big] $}
\\[1ex]
= \IP\Big[A_{\ov{m}_2}\Big(\mu^\prime_{2,2} + \dsl_{1 \le \ell \le r} \rho^\ell_{2,1} + \rho^\ell_{1,2}\Big)\Big] \,,
\end{array}
\end{equation}
so that
\begin{equation}\label{3.51}
\begin{array}{l}
\IP[A_{\ov{m}_2}(\mu_{2,2})] \stackrel{(\ref{3.18})}{=} \IP [A_{\ov{m}_2}(\mu^\prime_{2,2} + \mu^*_{2,2})]\stackrel{(\ref{3.20})}{\le} \IP\Big[A_{\ov{m}_2}\Big(\mu^\prime_{2,2} + \dsl_{1 \le \ell \le r} \rho^\ell_{2,2}\Big)\Big]  \stackrel{(\ref{3.50})}{\le} 
\\
\IP\Big[A_{\ov{m}_2}\Big(\mu_{2,2}^\prime + \dsl_{1 \le \ell \le r} \rho^\ell_{2,1} + \rho^\ell_{1,2}\Big)\Big] \stackrel{(\ref{3.20})}{\le} \IP[A_{\ov{m}_2}(\mu^\prime_{2,2} + \mu^\prime_{2,1} + \mu^\prime_{1,2}), \ov{\rho}_{2,1} = 0 = \ov{\rho}_{1,2}]
\\[2ex]
+\; \IP[\ov{\rho}_{2,1} \;\mbox{or} \;\ov{\rho}_{1,2} \not= 0] = \IP [A_{\ov{m}_2}(\mu_{V,u^\prime}), \,\ov{\rho}_{2,1} = 0 = \ov{\rho}_{1,2} ]  +  \IP[\ov{\rho}_{2,1}  \;\mbox{or} \;\ov{\rho}_{1,2} \not= 0] 
\\[1ex]
 \le p_n(u^\prime) + 1 - e^{-\ov{\xi}_{2,1}(W_+)} + 1 - e^{-\ov{\xi}_{1,2}(W_+)} \stackrel{(\ref{3.24}), (\ref{3.25})}{\le} p_n(u^\prime) + 2 \,u^\prime\,c^r \,L_n^{(d-2) - a(d-2)r}\,.
\end{array}
\end{equation}

\medskip\n
This is the promised upper bound on $\IP[A_{\ov{m}_2}(\mu_{2,2})]$. We can now come back to (\ref{3.13}), (\ref{3.17}) and obtain that when $L_0$ is large for $n \ge 0$, $r \ge 1$, $0 < u^\prime < u$ satisfying (\ref{3.45}) one has
\begin{equation}\label{3.52}
p_{n+1}(u) \le c_2 \,\ell_n^{2(d-1)} p_n(u) \big(p_n(u^\prime) + u^\prime \,c^r_3 \,L_n^{(d-2)(1-ar)}\big)\,.
\end{equation}

\n
Given $u_0 > 0$, $r \ge 1$, we thus define the increasing sequence
\begin{equation}\label{3.53}
u_{n+1} = \Big(1 + \dis\frac{c_1}{\ell_n^{(d-2)}}\Big)^{r+1} \,u_n, \;\mbox{for $n \ge 0$}\,,
\end{equation}
as well as the sequence
\begin{equation}\label{3.54}
a_n = c_2 \, \,\ell_n^{2(d-1)} p_n(u_n), \;n \ge 0\,.
\end{equation}

\n
We will now prove a lemma that uses inequality (\ref{3.52}) to set-up an induction scheme ensuring that $a_n$ is at most $L_n^{-1}$ for all $n \ge 0$. Note that (\ref{3.52}) deteriorates when 
$u^\prime$ becomes large. This is compensated by picking $r$ sufficiently big and checking (\ref{3.56})ii)
at each step. In the end to be able to initiate the induction, we will need to pick $L_0$ large, then $ u_0 \ge c(L_0)$ to check (\ref{3.56})i) for $n=0$, and finally $r \ge c(L_0, u_0)$, see (\ref{3.65}), thus influencing the whole sequence $u_n$, so as to ensure that  (\ref{3.56})ii) holds for $n=0$.
\begin{lemma}\label{lem3.4}
If $L_0 \ge c$, then for $r$ such that
\begin{equation}\label{3.55}
(d-2) \,ar \ge 4d\,,
\end{equation}

\n
and any $u_0 > 0$, when for some $n \ge 0$,
\begin{equation}\label{3.56}
\begin{array}{lll}
{\rm i)}  \;\; a_n \le L_n^{-1}, &\mbox{and} \;&{\rm ii)}  \;\;u_n \le L_n^{(d-2) \,a \frac{r}{2}}\;,
\end{array}
\end{equation}

\n
then {\rm (\ref{3.56})} holds as well with $n+1$ in place of $n$.
\end{lemma}

\begin{proof}
Since $p_n(\cdot)$ is a non-increasing function, we see from (\ref{3.52}) that for $u_0 > 0$, $r \ge 1$, one has for $n \ge 0$,
\begin{equation*}
a_{n+1} \le a_n \Big(\Big(\dis\frac{\ell_{n+1}}{\ell_n}\Big)^{2(d-1)} a_n + c_2 \,\ell_{n+1}^{2(d-1)} \;u_n \,c_3^r \,L_n^{(d-2)(1-ar)}\Big)\,,
\end{equation*}
and since one also has
\begin{equation}\label{3.57}
\dis\frac{\ell_{n+1}}{\ell_n} \stackrel{(\ref{3.2})}{=} \dis\frac{[L^a_{n+1}]}{[L^a_n]} \le c\,\Big(\dis\frac{L_{n+1}}{L_n}\Big)^a \stackrel{(\ref{3.2})}{=} c\,\ell^a_n \le c\,L^{a^2}_n, \;\mbox{for $n \ge 0$}\,,
\end{equation}
we find:
\begin{equation}\label{3.58}
a_{n+1} \le c_4 \,a_n \big(L_n^{2(d-1) a^2} a_n + L_n^{2(d-1)a(1+a)} u_n \,c_3^r \,L_n^{(d-2)(1-ar)}\big), \;\mbox{for $n \ge 0$}\,.
\end{equation}

\n
We will now seek to propagate (\ref{3.56}) i) from $n$ to $n+1$. For this purpose it suffices to show that the following two inequalities hold:
\begin{align}
&c_4\,L_n^{2(d-1)a^2} a_n \le \fr \;\mbox{\f $\dis\frac{L_n}{L_{n+1}}$}, \;\mbox{and} \label{3.59}
\\[2ex]
&c_4\,u_n\,c^r_3 \, L_n^{2(d-1)a(1+a) + (d-2)(1-ar)}   \le \fr \;\mbox{\f $\dis\frac{L_n}{L_{n+1}}$}\,.\label{3.60}
\end{align}

\n
To check (\ref{3.59}) observe that:
\begin{equation}\label{3.61}
c_4 \,L_n^{2(d-1) a^2} a_n \stackrel{(\ref{3.56}) i)}{\le} c_4 \, L_n^{2(d-1)a^2 -1} \stackrel{(\ref{3.1})}{\le} c_4 \,L_n^{a-1} \le \mbox{\f $\dis\frac{1}{200}$} \;L_n^{-a} \stackrel{(\ref{3.2})}{\le} \fr  \;\mbox{\f $\dis\frac{L_n}{L_{n+1}}$}\;,
\end{equation}

\medskip\n
with $L_0 \ge c$ and (\ref{3.1}) in the next to last inequality.

\medskip
We now turn to (\ref{3.60}) and observe that when $r$ satisfies (\ref{3.55}) then
\begin{equation}\label{3.62}
\begin{array}{lcl}
u_n \,c_3^r \,L_n^{(d-2)(1-ar)} &\hspace{-2ex} \stackrel{(\ref{3.56}) ii)}{\le} &\hspace{-2ex} c_3^r \,L_n^{(d-2)(1-a\frac{r}{2})} \stackrel{(\ref{3.55})}{\le} \big(c_3 \,L_n^{-(d-2)\frac{a}{4}}\Big)^r \,L^{-2}_n
\\[1ex]
& \hspace{-2ex} \le &\hspace{-2ex}  L^{-2}_n , \;\mbox{if $L_0 \ge c$}\,.
\end{array}
\end{equation}

\n
As a result we see that the left-hand side of (\ref{3.60}) is smaller than 
\begin{equation*}
c_4\,L_n^{2(d-1)a(1+a) -2} \stackrel{(\ref{3.1})}{\le} c_4\,L_n^{-1} \le \mbox{\f $\dis\frac{1}{200}$} \;L_n^{-a} \le \fr \;\mbox{\f $\dis\frac{L_n}{L_{n+1}}$}, \;\mbox{if $L_0 \ge c$}\,.
\end{equation*}

\n
Recalling (\ref{3.61}), we see that for large $L_0$ we can propagate (\ref{3.56}) i) from $n$ to $n+1$. We now turn to (\ref{3.56}) ii). We have with (\ref{3.53}):
\begin{equation}\label{3.63}
\begin{split}
u_{n+1} & = \Big(1 + \dis\frac{c_1}{\ell_n^{d-2}}\Big)^{r+1} u_n \stackrel{(\ref{3.56}) ii)}{\le} \Big(1 + \dis\frac{c_1}{\ell_n^{d-2}}\Big)^{r+1} \,L_n^{(d-2) a \frac{r}{2}}
\\[1ex]
& = L_{n+1}^{(d-2)a \frac{r}{2}} \,\ell_n^{-(d-2) a \frac{r}{2}} \Big(1 + \dis\frac{c_1}{\ell_n^{d-2}}\Big)^{r+1}  \stackrel{r \ge 1}{\le} L_{n+1}^{(d-2) a \frac{r}{2}} \Big[\Big( 1 + \dis\frac{c_1}{\ell_n^{d-2}}\Big)^2 \,\ell_n^{-(d-2) \frac{a}{2}}\Big]^r
\\[1ex]
&\hspace{-0.8ex} \stackrel{(\ref{3.2})}{\le} L_{n+1}^{(d-2) a \frac{r}{2}} \big[(1 + c_1)^2 (100[L^a_0])^{-(d-2)\frac{a}{2}}\big]^r \le L_{n+1}^{(d-2)a \frac{r}{2}}, \;\mbox{if $L_0 \ge c$}\,.
\end{split}
\end{equation}

\medskip\n
Hence for $L_0 \ge c$, we can propagate (\ref{3.56}) ii) from $n$ to $n+1$ as well, and this concludes the proof of Lemma \ref{lem3.4}.
\end{proof}

We now choose $L_0$ large so that for any $u_0 > 0$ and $r \ge 1$ satisfying (\ref{3.55}), when (\ref{3.56}) holds for $n = 0$, then it holds for all $n \ge 0$. If we now choose $u_0 \ge c(L_0)$, we see that for any $m \in I_0$, (recall $I_0$ is the set of labels at level $0$),
\begin{equation}\label{3.64}
\begin{split}
a_0 & \stackrel{(\ref{3.54})}{=} c_2 \,\ell_0^{2(d-1)} \,p_0(u_0) = c_2 \,\ell_0^{2(d-1)} \,\IP[A^{u_0}_m]
\\
&\;\;  \le c_2 \,\ell_0^{2(d-1)} \,\IP[\cV^{u_0} \cap \wt{S}_m \not= \phi] \stackrel{(\ref{1.58})}{\le} c_2 \,\ell_0^{2(d-1)} |\wt{S}_m| \,e^{-u_0/g(0)}
\\[1ex]
&\;\;  \le c \, \,L_0^{2(d-1)a+d-1}\,e^{-u_0/g(0)} \le L_0^{-1}\,,
\end{split}
\end{equation}

\n
using $u_0 \ge c(L_0)$ in the last step. Similarly given $L_0 \ge c$ and $u_0 \ge c(L_0)$ as above, we can pick $r \ge c(L_0,u_0)$ such that:
\begin{equation}\label{3.65}
u_0 \le L_0^{(d-2)a \frac{r}{2}}\;.
\end{equation}

\n
With such choices, as noted above, it follows that $a_n \le L_n^{-1}$, for all $n \ge 0$, and this completes the proof of Proposition \ref{prop3.1}.
 \end{proof}
 
 This now brings us to the main result of this section. We recall the definition of the critical value $u_*$ in (\ref{2.17a}).
 
 \begin{theorem}\label{theo3.5} $(d \ge 3)$
 
 \medskip
 For large $u$ the vacant set $\cV^u$ does not percolate, i.e.
 \begin{equation}\label{3.66}
 u_* < \infty\,,
 \end{equation}
 and for $u > u_*$, $\IP[{\rm Perc}(u)] = 0$.
 \end{theorem}
 
 \begin{proof}
 With Corollary \ref{cor2.3} we only need to prove (\ref{3.66}). We choose $L_0, u_0, r$ as in Proposition \ref{prop3.1}, so that with $u_n$ as in (\ref{3.9}) we find:
 \begin{equation}\label{3.67}
 c_2 \,\ell_n^{2(d-1)} \,\IP[A^{u_n}_m] \le L_n^{-1}, \;\mbox{for any $n \ge 0$ and $m \in I_n$}\,.
 \end{equation}
 
\n
 With (\ref{3.2}) we know that $L_n \ge L_0^{(1+a)^n}$, and hence $\sum_n \,\ell_n^{-(d-2)} < \infty$, and we thus see that
 \begin{equation}\label{3.68}
 u_\infty = u_0 \prod\limits_{n \ge 0} \;\Big(1 + \dis\frac{c_1}{\ell_n^{d-2}}\Big)^{r+1} = u_0 \Big(\prod\limits_{n \ge 0} \,\Big(1 + \dis\frac{c_1}{\ell_n^{d-2}}\Big)\Big)^{r+1} < \infty\,.
 \end{equation}
 
 \n
 Consequently for any $n \ge 0$, and $m \in I_n$ such that $0 \in C_m$, we find as a consequence of (\ref{2.2}) and (\ref{3.7}) that
 \begin{equation}\label{3.69}
 \eta(u_\infty) \le \IP [A_m^{u_\infty}] \le c\,L^{-1}_n\,.
 \end{equation}
 
\medskip \n
 Letting $n$ tend to infinity we see that $\eta(u_\infty) = 0$, and (\ref{3.66}) follows.
 \end{proof}

\begin{remark}\label{rem3.6} \rm Once we know that $\cV^u$ does not percolate for large $u$, it is a natural question to wonder how large the vacant cluster at the origin can be. An exponential tail bound on the number of sites of the vacant cluster at the origin of subcritical Bernoulli percolation is known to hold, cf.~\cite{Grim99}, p.~132 and 350. Such an estimate cannot be true in the case of $\cV^u$ due to (\ref{1.65a}). The exact nature of the tail of this random variable  is an interesting problem. \hfill $\square$
\end{remark}

\section{Percolation for small  $u$}
 \setcounter{equation}{0}
 
The main objective of this section is to show that when $d \ge 7$, the vacant set $\cV^u$ percolates for small $u > 0$, or equivalently that $u_* > 0$, see Theorem \ref{theo4.3}. In spite of the fact that $\cV^u$ tends to contain bigger boxes than what is the case for Bernoulli percolation, cf.~(\ref{1.65a}), it does not stochastically dominate Bernoulli percolation in the highly percolative regime as noted in Remark \ref{rem1.6} 1). This fact precludes a strategy based on a direct comparison argument. We develop here a similar but simpler renormalization procedure as in the previous section. It yields a sharper result than the strategy based on the combination of the exponential bound (\ref{2.23}) and a Peierls-type argument, as outlined in Remark 3.5 3). Such a proof only works for $d \ge 18$.
 
 \medskip
 We begin with some notation. We recall the definition of $a > 0$ and $L_n, n \ge 0$, in (\ref{3.1}), (\ref{3.2}). Throughout we identify $\IZ^2$ with the subset of points $z = (z_1,\dots,z_d)$ in $\IZ^d$, such that $z_3 = z_4 = \dots = z_d = 0$. For $n \ge 0$, we define the set $J_n$ of labels of level $n$ just as in (\ref{3.4}), but with $d$ replaced by $2$. For $n \ge 0$ and $m \in J_n$ we attach the boxes in $\IZ^2$, $D_m \subseteq \wt{D}_m$, with a similar definition as in (\ref{3.5}), but with $d$ replaced by $2$, and $I_n$ by $J_n$. We write $\wt{V}_m$ for the relative interior boundary in $\IZ^2$ of $\wt{D}_m$, i.e. the set of points of $\wt{D}_m$ neighboring $\IZ^2 \backslash \wt{D}_m$. In this section, parallel to (\ref{3.7}), a crucial role is played by the ``occupied crossing'' events. Namely for $u \ge 0, n\ge 0, m \in J_n$, we define:
 \begin{equation}\label{4.1}
 \mbox{$B^u_m = \{\o \in \Omega$; there is a $*$-nearest neighbor path in $\cI^u(\o)\cap \wt{D}_m$ from $D_m$ to $\wt{V}_m\}$}, 
 \end{equation}
As a consequence of translation invariance, cf.~Theorem \ref{theo2.1} or (\ref{1.48}),
\begin{equation}\label{4.2}
q_n(u) = \IP[B^u_m], \;u \ge 0, n\ge 0, \;\mbox{with} \; m \in J_n\,,
\end{equation}

\n
is well defined. It is also a non-decreasing function of $u$. Our main tasks consists in showing that when $u$ is chosen small $q_n(u)$ tends sufficiently rapidly to $0$. Our key control stems from the
\begin{proposition}\label{prop4.1} $(d \ge 7)$

\medskip
There exists a positive constant $c_5$, cf.~{\rm (\ref{4.13})}, such that for $L_0 \ge c$, and $u \le c(L_0)$ one has
\begin{equation}\label{4.3}
c_5 \,\ell^2_n \,q_n(u) \le L_n^{-\frac{1}{2}}, \; \mbox{for all $n \ge 0$}\,.
\end{equation}
\end{proposition}

\begin{proof}
In analogy with (\ref{3.11}), (\ref{3.12}), we define for $n \ge 0$ and $m \in J_{n+1}$ the collection of labels of boxes at level $n$ ``at the boundary of $D_m$'':
\begin{equation}\label{4.4}
\mbox{$\cK_1 =  \{\ov{m} \in J_n; \,D_{\ov{m}} \subseteq D_m$ and some point of $D_{\ov{m}}$ neighbors  $\IZ^2 \backslash D_m\}$}\,.
\end{equation}

\medskip\n
We also consider the collection of labels of boxes at level $n$ containing some point at $|\cdot|_\infty$-distance $\frac{L_{n+1}}{2}$ from $D_m$:
\begin{equation}\label{4.5}
\cK_2 = \Big\{\ov{m} \in J_n; \, D_{\ov{m}} \cap \Big\{z \in \IZ^2; \,d(z,D_m) = \mbox{\f $\dis\frac{L_{n+1}}{2}$} \Big\} \not= \phi\Big\}\,.
\end{equation}

\n
The argument leading to (\ref{3.13}) applies here as well and shows that
\begin{equation}\label{4.6}
q_{n+1}(u) \le c\,\ell^2_n \;\sup\limits_{\ov{m}_1 \in \cK_1, \ov{m}_2 \in \cK_2} \IP[B^u_{\ov{m}_1} \cap B^u_{\ov{m}_2}], \;\mbox{for $u \ge 0$}\,.
\end{equation}

\n
From now on we assume that
\begin{equation}\label{4.7}
u \le 1\,.
\end{equation}

\n
For $\ov{m}_1 \in \cK_1$, $\ov{m}_2 \in \cK_2$, we define $V = \wt{D}_{\ov{m}_1} \cup \wt{D}_{\ov{m}_2}$ and write
\begin{equation}\label{4.8}
\mu_{V,u} = \delta_{1,1} + \delta_{1,2} + \delta_{2,1} + \delta_{2,2}\,,
\end{equation}

\medskip\n
with a similar definition as in (\ref{3.14}) or (\ref{2.8}), so that the $\delta_{i,j}(dw)$, $ 1 \le i, j \le 2$, are independent Poisson point processes on $W_+$ with respective intensity measures $\zeta_{i,j}, 1 \le i, j \le 2$, given by analogous formulas as in (\ref{2.10}), except that $K$ and $K + x$ are now respectively replaced by $\wt{D}_{\ov{m}_1}$ and $\wt{D}_{\ov{m}_2}$. With a similar notation as in (\ref{3.15}), when $\Lambda(dw)$ is a random point process on $W_+$, we write:
\begin{equation}\label{4.9}
\begin{split}
B_{\ov{m}}(\Lambda) = \big\{\o \in \Omega; &\; \mbox{$D_{\ov{m}}$ and $\wt{V}_{\ov{m}}$ are connected by a $*$-nearest-neighbor path}
\\
&\;\mbox{in $\wt{D}_{\ov{m}} \cap \big(\bigcup\limits_{w \in {\rm Supp}(\Lambda(\o))} w(\IN)\big)\big\}, \;\ov{m} \in J_n$}\,.
\end{split}
\end{equation}

\n
For instance we now see with (\ref{1.54}) that
\begin{equation}\label{4.10}
B^u_{\ov{m}} = B_{\ov{m}}(\mu_{K,u}) \;\mbox{for any $K \supseteq \wt{D}_{\ov{m}}, \;\ov{m} \in J_n$}\,.
\end{equation}

\n
Specializing to $K = V$, and noting that
\begin{equation*}
\wt{D}_{\ov{m}_i} \cap \big(\bigcup\limits_{w \in {\rm Supp}(\delta_{j,j})} w(\IN)\big) = \phi, \;\mbox{when $\{i,j\} = \{1,2\}$}\,,
\end{equation*}
we obtain the identities:
\begin{equation}\label{4.11}
\begin{split}
B^u_{\ov{m}_1} & = B_{\ov{m}_1}(\mu_{V,u}) = B_{\ov{m}_1} (\delta_{1,1} + \delta_{1,2} + \delta_{2,1})\,,
\\[1ex]
B^u_{\ov{m}_2} & = B_{\ov{m}_2}(\mu_{V,u}) = B_{\ov{m}_2} (\delta_{2,2} + \delta_{2,1} + \delta_{1,2})\,.
\end{split}
\end{equation}

\n
Due to the independence of the $\delta_{i,j}$, $1 \le i,j \le 2$, it follows that for $\ov{m}_i \in \cK_i$, $i = 1,2$, we have:
\begin{equation}\label{4.12}
\begin{split}
&\IP[B^u_{\ov{m}_1} \cap B^u_{\ov{m}_2}] = \IP\big[B_{\ov{m}_1} (\delta_{1,1} + \delta_{1,2} + \delta_{2,1}) \cap B_{\ov{m}_2} (\delta_{2,2} + \delta_{2,1} + \delta_{1,2})\big]
\\[1ex]
&\le  \IP\big[B_{\ov{m}_1} (\delta_{1,1}) \cap B_{\ov{m}_2} (\delta_{2,2}) , \,\delta_{1,2} = \delta_{2,1} = 0\big] + \IP[\delta_{1,2} \;\mbox{or} \;\delta_{2,1} \not= 0]
\\[1ex]
& \le \IP[B_{\ov{m}_1}(\delta_{1,1})] \,\IP[B_{\ov{m}_2}(\delta_{2,2})] + \IP[\delta_{1,2} \not= 0] + \IP[\delta_{2,1} \not= 0]
\\[1ex]
&\hspace{-3ex} \stackrel{(\ref{4.2}), (\ref{4.11})}{\le} q_n(u)^2 + 1 - e^{-\zeta_{1,2}(W_+)} + 1 - e^{-\zeta_{2,1}(W_+)} 
\\[1ex]
&\le q_n(u)^2 + u\big(P_{e_V}[X_0 \in \wt{D}_{\ov{m}_1}, H_{\wt{D}_{\ov{m}_2}} < \infty] + P_{e_V} [X_0 \in \wt{D}_{\ov{m}_2}, H_{\wt{D}_{\ov{m}_1}} < \infty]\big)
\\[1ex]
&\hspace{-1ex} \stackrel{(\ref{4.7})}{\le}  q_n(u)^2 + c \,L^2_n \;\dis\frac{L^2_n}{L^{d-2}_{n+1}} \;,
\end{split}
\end{equation}

\n
where in the last step we used (\ref{1.6}), (\ref{1.8}) together with standard bounds on the Green function, cf.~\cite{Lawl91}, p.~31. With (\ref{4.6}), (\ref{4.12}), we thus see that
\begin{equation}\label{4.13}
q_{n+1}(u) \le c_5\,\ell^2_n \big(q^2_n(u) + L^4_n \,L^{-(d-2)}_{n+1} \big)\,.
\end{equation}
We thus define
\begin{equation}\label{4.14}
b_n = c_5 \,\ell^2_n \,q_n(u), \;\mbox{for $n \ge 0$}\,,
\end{equation}
and see that
\begin{equation*}
b_{n+1} \le \Big(\dis\frac{\ell_{n+1}}{\ell_n}\Big)^2 \;b^2_n + c(\ell_{n+1} \ell_n)^2 \,L^4_n \,L^{-(d-2)}_{n+1}, \;\mbox{for $n \ge 0$}\,.
\end{equation*}

\n
With (\ref{3.2}) we know that $(\ell_n \ell_{n+1})^2 \le c\,L^{2a}_n \;L^{2a}_{n+1} \le c\,L_n^{4a + 2a^2}$, and with (\ref{3.57}) we know that $\frac{\ell_{n+1}}{\ell_n} \le c\,L^{a^2}_n$. As a result we obtain:
\begin{equation}\label{4.15}
b_{n+1}  \le c (L^{2a^2}_n \;b^2_n + L_n^{-(d-2)(1+a) + 4 + 4a + 2a^2}) \stackrel{d \ge 7}{\le} c_6 \;L^{2a^2}_n \,(b^2_n + L^{-1}_n) \,.
\end{equation}
We will now use the following induction lemma.

\begin{lemma}\label{lem4.2} $(d \ge 7)$

\medskip
If $L_0 \ge c$, then for any $u \le 1$, when for some $n \ge 0$,
\begin{equation}\label{4.16}
b_n \le L_n^{-\frac{1}{2}}\;,
\end{equation}

\n
then {\rm (\ref{4.16})} holds as well with $n+1$ in place of $n$.
\end{lemma}

\begin{proof}
With (\ref{4.15}) we see that
\begin{equation}\label{4.17}
b_{n+1} \le 2 c_6 \,L_n^{2a^2 - 1} \stackrel{(\ref{3.2})}{\le} c\,L^{-\frac{1}{2}}_{n+1} \;L_n^{\frac{1}{2} \,(1 + a) + 2a^2 - 1} \stackrel{(\ref{3.1})}{\le} c\,L^{-\frac{1}{2}}_{n+1} \;L_0^{-\frac{1}{4}} \le L_{n+1}^{-\frac{1}{2}} \;,
\end{equation}

\medskip\n
when $L_0 \ge c$. This proves our claim.
\end{proof}

We now choose $L_0 \ge c$, so that for any $u \le 1$, when $b_0 \le L_0^{-\frac{1}{2}}$ holds then $b_n \le L_n^{-\frac{1}{2}}$ for all $n \ge 0$. Further picking $u \le c(L_0) (\le 1)$, we see that for $m \in J_0$,
\begin{equation}\label{4.18}
\begin{split}
b_0 \stackrel{(\ref{4.14})}{=} c_5 \,\ell^2_0 \,q_0(u) & \le c_5\, \ell^2_0 \,\IP[\cI^u \cap \wt{D}_m \not= \phi] \le c\,L^{2(1+a)}_0 \IP[0 \in \cI^u]
\\[1ex]
&\hspace{-1ex} \stackrel{(\ref{1.58})}{=} c\, L^{2(1+a)}_0 \,(1 - e^{-u/g(0)}) \le L_0^{-\frac{1}{2}}\,.
\end{split}
\end{equation}

\n
With this choice of $L_0$ and $u$ we thus find that $b_n \le L_n^{-\frac{1}{2}}$ for all $n \ge 0$, and this concludes the proof of Proposition \ref{prop4.1}.
\end{proof}

\begin{theorem}\label{theo4.3} $(d \ge 7)$

\medskip
For small $u > 0$ the vacant set $\cV^u$ does percolate, i.e.
\begin{equation}\label{4.19}
u_* > 0\,,
\end{equation}
and for $u < u_*$, $\IP[{\rm Perc}(u)] = 1$.
\end{theorem}

\begin{proof}
We only need to prove (\ref{4.19}) thanks to Corollary \ref{cor2.3}. We choose $L_0$ large and $u \le c (L_0)$ so that (\ref{4.3}) holds for all $n \ge 0$. Then for $n_0 \ge 0$ and $M = L_{n_0} -1$, we can write:
\begin{equation}\label{4.20}
\begin{split}
1 - \eta(u)  \le&\; \IP[\mbox{$0$ does not belong to an infinite connected component of $\cV^u \cap \IZ^2]$}
\\
 \le &\; \IP[\cI^u \cap B(0,M) \cap \IZ^2 \not= \phi] + \IP[\cI^u \cap \big(\IZ^2 \backslash B(0,M)\big) \;\mbox{contains a}
\\[-0.5ex]
&\;\mbox{$*$-nearest neighbor circuit surrounding $0] \stackrel{(\ref{1.58})}{\le} c\,M^2 (1- e^{-u/g(0)}) \;+$}
\\
&\; \dsl_{n \ge n_0} \IP\big[\cI^u \cap \big(\IZ^2 \backslash B(0,M)\big) \;\mbox{contains a $*$-nearest neighbor circuit}
\\ 
&\;\mbox{containing $0$ and passing through a point in $[L_n, L_{n+1} - 1]\,e_1]$}
\\[1ex]
 \le & \;c\,L^2_{n_0} \,u + \dsl_{n \ge n_0} \;\dsl_m \;\IP[B^u_m]\,,
\end{split}
\end{equation}

\n
where $m$ runs over the collection of labels at level $n$ of boxes $D_m$ intersecting the segment $[L_n, L_{n+1} - 1] \,e_1$, ($(e_1,...,e_d)$ stands for the canonical basis of $\IR^d$, and recall we identified $\IZ^2$ with $\IZ e_1 + \IZ e_2$). With (\ref{3.2}) this collection has cardinality at most $\ell_n \le c\,L_n^a$,  and we thus find that
\begin{equation}\label{4.21}
1 - \eta(u) \le c\,L^2_{n_0} \,u + \dsl_{n \ge n_0} \,c\,L^a_n \,L_n^{-\frac{1}{2}} \;\stackrel{(\ref{3.1})}{\le} c(L^2_{n_0} \,u + \dsl_{n \ge n_0} \,L_n^{-\frac{1}{4}}) \,.
\end{equation}

\n
Choosing $n_0$ large and then $u \le c(L_0, n_0)$, we find that $1 - \eta(u) < 1$, and this proves (\ref{4.19}). 
\end{proof}

\begin{remark}\label{rem4.4} ~ \rm

\medskip\n
1) Combining Theorems \ref{theo3.5} and \ref{theo4.3} we find that when $d \ge 7$,
\begin{equation}\label{4.22}
0 < u_* < \infty\,,
\end{equation}
i.e. $u_*$ is a non-degenerate critical value. This has been extended to all $d \ge 3$ in \cite{SidoSzni08}.

\bigskip\n
2) As a matter of fact the above proof combined with an ergodicity argument, cf. (\ref{2.6}), shows that when $d \ge 7$, for small $u > 0$, $\cV^u$ percolates in  $\IZ^2 \subset \IZ^d$.  This feature remains true for all $d \ge 3$, see Theorem 3.4 of \cite{SidoSzni08}.

\bigskip\n
3) Some other very natural questions remain open. When $u > u_*$, how large is the vacant cluster at the origin? When $u < u_*$, how large can a vacant cluster at the origin be, if it does not meet the infinite cluster, (which is known to be unique thanks to \cite{Teix08})? Is there percolation at criticality (i.e. is $\eta(u_*) > 0$)? What is the asymptotic behavior of $u_*$ for large $d$? These are just a few examples of many unresolved issues concerning percolative properties of the vacant set left by the random interlacements model described in this work. \hfill $\square$
\end{remark}

\end{document}